\documentclass[10pt]{amsart}
\usepackage[latin9]{inputenc}
\usepackage{color}
\usepackage{amstext}
\usepackage{amsthm}
\usepackage{amssymb}
\usepackage{esint}

\makeatletter
\numberwithin{equation}{section}
\numberwithin{figure}{section}
\theoremstyle{plain}
\newtheorem{thm}{\protect\theoremname}[section]
\theoremstyle{definition}
\newtheorem{defn}[thm]{\protect\definitionname}
\theoremstyle{remark}
\newtheorem{rem}[thm]{\protect\remarkname}
\theoremstyle{plain}
\newtheorem{lem}[thm]{\protect\lemmaname}
\theoremstyle{plain}
\newtheorem{prop}[thm]{\protect\propositionname}
\theoremstyle{remark}
\newtheorem*{acknowledgement*}{\protect\acknowledgementname}

\usepackage{color}
\usepackage{latexsym}
\usepackage{bm}
\usepackage{graphicx}
\usepackage{wrapfig}
\usepackage{fancybox}
\usepackage{hyperref}
\usepackage{bbm}
\usepackage{arydshln}

     \def\section{\@startsection{section}{1}%
     \z@{.7\linespacing\@plus\linespacing}{.5\linespacing}%
     {\bfseries
     \centering
     }}
     \def\@secnumfont{\bfseries}

\newcommand{\Rd}{\mathbb{R}^{d}}

  \providecommand{\lemmaname}{Lemma}
  \providecommand{\propositionname}{Proposition}
  \providecommand{\remarkname}{Remark}
\providecommand{\theoremname}{Theorem}

  \providecommand{\lemmaname}{Lemma}
  \providecommand{\propositionname}{Proposition}
  \providecommand{\remarkname}{Remark}
\providecommand{\theoremname}{Theorem}

  \providecommand{\lemmaname}{Lemma}
  \providecommand{\propositionname}{Proposition}
  \providecommand{\remarkname}{Remark}
\providecommand{\theoremname}{Theorem}

  \providecommand{\lemmaname}{Lemma}
  \providecommand{\propositionname}{Proposition}
  \providecommand{\remarkname}{Remark}
\providecommand{\theoremname}{Theorem}

\makeatother

\providecommand{\acknowledgementname}{Acknowledgement}
\providecommand{\definitionname}{Definition}
\providecommand{\lemmaname}{Lemma}
\providecommand{\propositionname}{Proposition}
\providecommand{\remarkname}{Remark}
\providecommand{\theoremname}{Theorem}

\begin{document}
\title[Limiting distribution for affine processes]{Existence of limiting distribution for affine processes}
\author[Peng Jin]{Peng Jin\textsuperscript{*}}
\thanks{\textsuperscript{*}Peng Jin is partially supported by the STU Scientific Research Foundation for Talents (No. NTF18023).}
\author{Jonas Kremer}
\author{Barbara R\"udiger}

\address[Peng Jin]{Department of Mathematics \\ Shantou University \\ Shantou, Guangdong 515063, China}

\email{pjin@stu.edu.cn}

\address[Jonas Kremer]{Fakult\"at f\"ur Mathematik und Naturwissenschaften\\ Bergische Universit\"at Wuppertal\\ 42119 Wuppertal, Germany}

\email{jkremer@uni-wuppertal.de}

\address[Barbara R\"udiger]{Fakult\"at f\"ur Mathematik und Naturwissenschaften\\ Bergische Universit\"at Wuppertal\\ 42119 Wuppertal, Germany}

\email{ruediger@uni-wuppertal.de}

\date{\today}



\begin{abstract}In this paper, sufficient conditions are given for
the existence of limiting distribution of a conservative affine process
on the canonical state space $\mathbb{R}_{\geqslant0}^{m}\times\mathbb{R}^{n}$,
where $m,\thinspace n\in\mathbb{Z}_{\geqslant0}$ with $m+n>0$. Our
main theorem extends and unifies some known results for OU-type processes
on $\mathbb{R}^{n}$ and one-dimensional CBI processes (with state
space $\mathbb{R}_{\geqslant0}$). To prove our result, we combine
analytical and probabilistic techniques; in particular, the stability
theory for ODEs plays an important role. \end{abstract}

\maketitle

\allowdisplaybreaks

\section{Introduction}

Let $D:=\mathbb{R}_{\geqslant0}^{m}\times\mathbb{R}^{n}$, where $m,n\in\mathbb{Z}_{\geqslant0}$
with $m+n>0$. Roughly speaking, an affine process with state space
$D$ is a time-homogeneous Markov process $(X_{t})_{t\geqslant0}$
taking values in $D$, whose $\log$-characteristic function depends
in an affine way on the initial value of the process, that is, there
exist functions $\phi$, $\psi=(\psi_{1},\ldots,\psi_{m+n})$ such
that
\[
\mathbb{E}\left[\left.\mathrm{e}^{\langle u,X_{t}\rangle}\thinspace\right|\thinspace X_{0}=x\right]=\mathrm{e}^{\phi(t,u)+\langle\psi(t,u),x\rangle},
\]
for all $u\in\mathrm{i}\mathbb{R}^{m+n}$, $t\geqslant0$ and $x\in D$.
The general theory of affine processes was initiated by Duffie, Pan
and Singleton \cite{MR1793362} and further developed by Duffie, Filipovi{\'c},
and Schachermayer \cite{MR1994043}. In the seminal work of Duffie
\textit{et al.} \cite{MR1994043}, several fundamental properties
of affine processes on the canonical state space $D$ were established.
In particular, the generator of $D$-valued affine processes is completely
characterized through a set of \textit{admissible parameters, }and
the associated \textit{generalized Riccati equations} for $\phi$
and $\psi$ are introduced and studied.\textit{ } The results of
\cite{MR1994043} were further complemented by many subsequent developments,
see, e.g., \cite{MR2284011,MR3340375,MR3540486,MR2243880,MR2648460,MR2931348,MR3313754,MR2851694}.

Affine processes have found a wide range of applications in finance,
mainly due to their computational tractability and modeling flexibility.
Many popular models in finance, such as the models of Cox \textit{et
al.} \cite{MR785475}, Heston \cite{Heston} and Vasicek \cite{MR3235239},
are of affine type. Moreover, from the theoretical point of view,
the concept of affine processes enables a unified treatment of two
very important classes of continuous-time Markov processes: OU-type
processes on $\mathbb{R}^{n}$ and CBI (continuous-state branching
processes with immigration) processes on $\mathbb{R}_{\geqslant0}^{m}$.

In this paper, we are concerned with the following question: when
does an affine process converge in law to a limit distribution? This
problem has already been dealt with in the following situations:
\begin{itemize}
\item Sato and Yamazato \cite{MR738769} provided conditions under which
an OU-type process on $\mathbb{R}^{n}$ converges in law to a limit
distribution, and they identified this type of limit distributions
with the class of operator self-decomposable distributions of Urbanik
\cite{MR0245068};
\item without a proof, Pinsky \cite{MR0295450} announced the existence
of a limit distribution for one-dimensional CBI processes, under a
mean-reverting condition and the existence of the log-moment of the
Lévy measure from the immigration mechanism. A recent proof appeared
in \cite[Theorem 3.20 and Corollary 3.21]{MR2760602} (see also \cite[Theorem 3.16]{MR2779872}).
A stronger form of this result can be found in \cite[Theorem 2.6]{MR2922631};
\item Glasserman and Kim \cite{MR2599675} proved that  affine diffusion
processes on $\mathbb{R}_{\geqslant0}^{m}\times\mathbb{R}^{n}$ introduced
by Dai and Singleton \cite{10.2307/222481} have limiting stationary
distributions and characterized these limits;
\item Barczy, Döring, Li, and Pap \cite{MR3254346} showed stationarity
of an affine two-factor model on $\mathbb{R}_{\geqslant0}\times\mathbb{R}$,
with one component being the $\alpha$-root process.
\end{itemize}
Our motivation for this paper is twofold. On the one hand, we would
like to formulate a general result for affine processes with state
space $D=\mathbb{R}_{\geqslant0}^{m}\times\mathbb{R}^{n}$, which
unifies the above mentioned results; on the other hand, our result
should also provide new results for the unsolved cases where $D=\mathbb{R}_{\geqslant0}^{m}$
($m\geqslant2$) and $D=\mathbb{R}_{\geqslant0}^{m}\times\mathbb{R}^{n}$
($m\geqslant1,n\geqslant1$). As our main result (see Theorem \ref{thm:stationarity of affine processes}
below), we give sufficient conditions such that an affine process
$X$ with state space $D=\mathbb{R}_{\geqslant0}^{m}\times\mathbb{R}^{n}$
converges in law to a limit distribution as time goes to infinity,
and we also identify this limit through its characteristic function.
Using a similar argument as in \cite{MR2779872}, we will show that
the limit distribution is the unique stationary distribution for $X$.
\\

The rest of this paper is organized as follows. In Section 2 we recall
some definitions regarding affine processes and present our main theorem,
whose proof we defer to Section 4. In Section 3 we deal with the large
time behavior of the function $\psi$ and show that $\psi(t,u)$ converges
exponentially fast to $0$ as $t$ goes to infinity. Finally, we prove
our main theorem in Section 4.

\section{Preliminaries and main result}

\subsection{Notation}

Let $\mathbb{N}$, $\mathbb{Z}_{\geqslant0}$, $\mathbb{R}$ denote
the sets of positive integers, non-negative integers and real numbers,
respectively. Let $\mathbb{R}^{d}$ be the $d$-dimensional ($d\geqslant1$)
Euclidean space and define
\[
\mathbb{R}_{\geqslant0}^{d}:=\left\{ x\in\mathbb{R}^{d}\thinspace:\thinspace x_{i}\geqslant0,\ i=1,\ldots,d\right\}
\]
and
\[
\mathbb{R}_{>0}^{d}:=\left\{ x\in\mathbb{R}^{d}\thinspace:\thinspace x_{i}>0,\ i=1,\ldots,d\right\} .
\]
For $x,\thinspace y\in\mathbb{R}$, we write $x\wedge y:=\min\{x,y\}$.
By $\langle\cdot,\cdot\rangle$ and $\Vert x\Vert$ we denote the
inner product on $\mathbb{R}^{d}$ and the induced Euclidean norm
of a vector $x\in\mathbb{R}^{d}$, respectively. For a $d\times d$-matrix
$A=(a_{ij})$, we write $A^{\top}$ for the transpose of $A$ and
define $\|A\|:=(\mathrm{trace}(A^{\top}A))^{1/2}$. Let $\mathbb{C}^{d}$
be the space that consists of $d$-tuples of complex numbers. We define
the following subsets of $\mathbb{C}^{d}$:
\[
\mathbb{C}_{\leqslant0}^{d}:=\left\{ u\in\mathbb{C}^{d}\thinspace:\thinspace\mathrm{Re}\thinspace u_{i}\leqslant0,\ i=1,\ldots,d\right\}
\]
 and
\[
\mathrm{i}\mathbb{R}^{d}:=\left\{ u\in\mathbb{C}^{d}\thinspace:\thinspace\mathrm{Re}\,u_{i}=0,\ i=1,\ldots,d\right\} .
\]

The following sets of matrices are of particular importance in this
work :
\begin{itemize}
\item $\mathbb{M}_{d}^{-}$ which stands for the set of real $d\times d$
matrices all of whose eigenvalues have strictly negative real parts.
Note that $A\in\mathbb{M}_{d}^{-}$ if and only if $\Vert\exp\left\{ tA\right\} \Vert\to0$
as $t\to\infty$;
\item $\mathbb{S}_{d}^{+}$ (resp. $\mathbb{S}_{d}^{++}$) which stands
for the set of all symmetric and positive semidefinite (resp. positive
definite) real $d\times d$ matrices.
\end{itemize}
If $A=(a_{ij})$ is a $d\times d$-matrix, $b=(b_{1},\ldots,b_{d})\in\mathbb{R}^{d}$
and $\mathcal{I},\thinspace\mathcal{J}\subset\{1,\ldots,d\}$, we
write $A_{\mathcal{I}\mathcal{J}}:=(a_{ij})_{i\in\mathcal{I},j\in\mathcal{J}}$
and $b_{\mathcal{I}}:=(b_{i})_{i\in\mathcal{I}}$.

Let $U$ be an open set or the closure of an open set in $\mathbb{R}^{d}$.
We introduce the following function spaces: $C^{k}(U)$, $C_{c}^{k}(U)$,
and $C^{\infty}(U)$ which denote the sets of\textcolor{red}{{} $\mathbb{{\normalcolor C}}$}-valued
functions on $U$ that are $k$-times continuously differentiable,
that are $k$-times continuously differentiable with compact support,
and that are smooth, respectively. The Borel $\sigma$-Algebra on
$U$ will be denoted by $\mathcal{B}(U)$.

Throughout the rest of this paper, let $D:=\mathbb{R}_{\geqslant0}^{m}\times\mathbb{R}^{n}$,
where $m,\thinspace n\in\mathbb{Z}_{\geqslant0}$ with $m+n>0$. Note
that $m$ or $n$ may be $0$. The set $D$ will act as the state
space of affine processes we are about to consider. The total dimension
of $D$ is denoted by $d=m+n$. We write $\mathcal{B}_{b}(D)$ for
the Banach space of bounded real-valued Borel measurable functions
$f$ on $D$ with norm $\Vert f\Vert_{\infty}:=\sup_{x\in D}|f(x)|$.

For $D$, we write
\[
I=\{1,\ldots,m\}\quad\text{and}\quad J=\{m+1,\ldots,m+n\}
\]
for the index sets of the $\mathbb{R}_{\geqslant0}^{m}$-valued components
and the $\mathbb{R}^{n}$-valued components, respectively. Define
\[
\mathcal{U}:=\mathbb{C}_{\leqslant0}^{m}\times\mathrm{i}\mathbb{R}^{n}=\left\{ u\in\mathbb{C}^{d}\,:\,\mathrm{Re}\,u_{I}\leqslant0,\quad\mathrm{Re}\,u_{J}=0\right\} .
\]
Note that $\mathcal{U}$ is the set of all $u\in\mathbb{C}^{d}$,
for which $x\mapsto\exp\left\{ \langle u,x\rangle\right\} $ is a
bounded function on $D$.

Further notation is introduced in the text.

\subsection{Affine processes on the canonical state space}

Affine processes on the canonical state space $D=\mathbb{R}_{\geqslant0}^{m}\times\mathbb{R}^{n}$
have been systematically studied in the well-known work \cite{MR1994043}.
We remark that affine processes considered in \cite{MR1994043} are
in full generality and are allowed to have explosions or killings.
In contrast to \cite{MR1994043}, in this paper we restrict ourselves
to \emph{conservative affine processes}. In terms of terminology and
notation, we mainly follow, instead of \cite{MR1994043}, the paper
by Keller-Ressel and Mayerhofer \cite{MR3313754}, where only the
conservative case was considered. \\

Let us start with a time-homogeneous and conservative Markov process
with state space $D$ and semigroup $(P_{t})$ acting on $\mathcal{B}_{b}(D)$,
that is,
\[
P_{t}f(x)=\int_{D}f(\xi)p_{t}(x,d\xi),\quad f\in\mathcal{B}_{b}(D).
\]
Here $p_{t}(x,\cdot)$ denotes the transition kernel of the Markov
process. We assume that $p_{0}(x,\{x\})$=1 and $p_{t}(x,D)\text{=}1$
for all $t\geqslant0,$ $x\in D$.

Let $(X,(\mathbb{P}_{x})_{x\in D})$ be the canonical realization
of $(P_{t})$  on $(\Omega,\mathcal{F},(\mathcal{F}_{t})_{t\geqslant0})$,
where $\Omega$ is the set of all càdlàg paths in $D$ and $X_{t}(\omega)=\omega(t)$
for $\omega\in\Omega$. Here $(\mathcal{F}_{t})_{t\geqslant0}$ is
the filtration generated by $X$ and $\mathcal{F}=\bigvee_{t\geqslant0}\mathcal{F}_{t}$.
The probability measure $\mathbb{P}_{x}$ on $\Omega$ represents
the law of the Markov process $(X_{t})_{t\geqslant0}$ started at
$x$, i.e., it holds that $X_{0}=x$, $\mathbb{P}_{x}$-almost surely.
The following definition is taken from \cite[Definition 2.2]{MR3313754}.
\begin{defn}
\label{def:(Affine-process)} The Markov process $X$ is called \textit{affine}
with state space $D$, if its transition kernel $p_{t}(x,A)=\mathbb{P}_{x}(X_{t}\text{\ensuremath{\in}}A)$
satisfies the following:

(i) it is stochastically continuous, that is, $\lim_{s\to t}{\displaystyle p_{s}(x,\cdot)=p_{t}(x,\cdot)}$
weakly for all $t\geqslant0,\;x\in D$, and

(ii) there exist functions $\phi:\mathbb{R}_{\geqslant0}\times\mathcal{U}\rightarrow\mathbb{C}$
and $\psi:\mathbb{R}_{\geqslant0}\times\mathcal{U}\rightarrow\mathbb{C}^{d}$
such that
\begin{equation}
{\displaystyle \int_{D}\mathrm{e}^{\langle u,\xi\rangle}p_{t}(x,\mathrm{d}\xi)=\mathbb{E}_{x}\left[\mathrm{e}^{\langle X_{t},u\rangle}\right]=\exp\left\{ \phi(t,u)+\langle x,\psi(t,u)\rangle\right\} }\label{eq:affine representation}
\end{equation}
for all $t\geqslant0,\;x\in D$ and $u\in\mathcal{U}$, where $\mathbb{E}_{x}$
denotes the expectation with respect to $\mathbb{P}_{x}$.

\end{defn}

The stochastic continuity in (i) and the affine property in (ii) together
imply the following regularity of the functions $\phi$ and $\psi$
(see \cite[Theorem 5.1]{MR2851694}), i.e., the right-hand derivatives
\begin{equation}
F(u):=\left.\frac{\partial}{\partial t}\phi(t,u)\right\vert _{t=0+}\quad\text{and}\quad R(u):=\left.\frac{\partial}{\partial t}\psi(t,u)\right\vert _{t=0+}\label{eq: defi F and R}
\end{equation}
exist for all $u\in\mathcal{U}$, and are continuous at $u=0$. Moreover,
according to \cite[Proposition 7.4]{MR1994043}, the functions $\phi$
and $\psi$ satisfy the \textit{semi-flow property}:
\begin{align}
\phi(t+s,u) & =\phi(t,u)+\phi\left(s,\psi(t,u)\right)\quad\text{and}\quad\psi(t+s,u)=\psi\left(s,\psi(t,u)\right),\label{eq:semi flow property}
\end{align}
for all $t,\thinspace s\geqslant0$ with $(t+s,u)\in\mathbb{R}_{\geqslant0}\times\mathcal{U}$.
\begin{defn}
\label{def:admissible parameters}We call $(a,\alpha,b,\beta,m,\mu)$
a \emph{set of admissible parameters} for the state space $D$ if

(i) $a\in\mathbb{S}_{d}^{+}$ and $a_{kl}=0$ for all $k\in I$ or
$l\in I$;

(ii) $\alpha=(\alpha_{1},\ldots,\alpha_{m})$ with $\alpha_{i}=(\alpha_{i,kl})_{1\leqslant k,l\leqslant d}\in\mathbb{S}_{d}^{+}$

\begin{flushright} and $\alpha_{i,kl}=0$ if $k\in I\backslash\{i\}$ or $l\in I\backslash\{i\}$;\end{flushright}

(iii) $m$ is a Borel measure on $D\backslash\{0\}$ satisfying

\[
\int_{D\backslash\{0\}}\left(1\wedge\left\Vert \xi\right\Vert ^{2}+\sum_{i\in I}\left(1\wedge\xi_{i}\right)\right)m(\mathrm{d}\xi)<\infty;
\]

(iv) $\mu=(\mu_{1},\ldots,\mu_{m})$ where every $\mu_{i}$ is a Borel
measure on $D\backslash\{0\}$ satisfying
\begin{equation}
\int_{D\backslash\{0\}}\left(\left\Vert \xi\right\Vert \wedge\left\Vert \xi\right\Vert ^{2}+\sum_{k\in I\backslash\{i\}}\xi_{k}\right)\mu_{i}\left(\mathrm{d}\xi\right)<\infty.\label{cond: conservative}
\end{equation}

(v) $b\in D$;

(vi) $\beta=(\beta_{ki})\in\mathbb{R}^{d\times d}$ with $\beta_{ki}-\int_{D\backslash\{0\}}\xi_{k}\mu_{i}(\mathrm{d}\xi)\geqslant0$
for all $i\in I$ and $k\in I\setminus\{i\}$,

\begin{flushright} and $\beta_{ki}=0$ for all  $k\in I$ and $i\in J$;\end{flushright}
\end{defn}

We remark that our definition of admissible parameters is a special
case of \cite[Definition 2.6]{MR1994043}, since we require here that
the parameters corresponding to killing are constant $0$; moreover,
the condition in (iv) is also stronger as usual, i.e., we assume that
the first moment of $\mu_{i}$'s exists, which, by \cite[Lemma 9.2]{MR1994043},
implies that the affine process under consideration is conservative.
However, we should remind the reader that \eqref{cond: conservative}
is not a necessary condition for conservativeness. In fact, an example
of a conservative affine process on $\mathbb{R}_{\geqslant0}$, which
violates \eqref{cond: conservative}, is provided in \cite[Section 3]{MR2763096}.
\\

We write $\psi=(\psi^{I},\psi^{J})\in\mathbb{C}^{m}\times\mathbb{C}^{n}$,
where $\psi^{I}=(\psi_{1},\ldots,\psi_{m})^{\top}$ and $\psi^{J}=(\psi_{m+1},\ldots,\psi_{m+n})^{\top}$.
Recall that $R=(R_{1},\ldots,R_{d})^{\top}:\mathcal{U}\to\mathbb{C}^{d}$
is given in (\ref{eq: defi F and R}). Define $R^{I}:=(R_{1},\ldots,R_{m})^{\top}:\mathcal{U}\to\mathbb{C}^{m}$.
For $u\in\mathcal{U}$, we will often write $u=(v,w)\in\mathbb{C}_{\leqslant0}^{m}\times\mathrm{i}\mathbb{R}^{n}$.\\

The next result is due to \cite[Theorem 2.7]{MR1994043}.
\begin{thm}
\label{thm:characterization of affine processes} Let $(a,\alpha,b,\beta,m,\mu)$
be a set of admissible parameters in the sense of Definition \ref{def:admissible parameters}.
Then there exists a (unique) conservative affine process $X$ with
state space $D$ such that its infinitesimal generator $\mathcal{A}$
operating on a function \textup{$f\in C_{c}^{2}(D)$ is given by}
\begin{align*}
\mathcal{A}f(x) & =\sum_{k,l=1}^{d}\left(a_{kl}+\sum_{i=1}^{m}\alpha_{i,kl}x_{i}\right)\frac{\partial^{2}f(x)}{\partial x_{k}\partial x_{l}}+\langle b+\beta x,\nabla f(x)\rangle\\
 & \quad+\int_{D\backslash\{0\}}\left(f\left(x+\xi\right)-f(x)-\langle\nabla_{J}f(x),\xi_{J}\rangle\mathbbm{1}_{\lbrace\Vert\xi\Vert\leqslant1\rbrace}\left(\xi\right)\right)m\left(\mathrm{d}\xi\right)\\
 & \quad+\sum_{i=1}^{m}x_{i}\int_{D\backslash\{0\}}\left(f\left(x+\xi\right)-f\left(x\right)-\langle\nabla f(x),\xi\rangle\right)\mu_{i}\left(\mathrm{d}\xi\right)
\end{align*}
where $x\in D$, $\nabla_{J}:=(\partial_{x_{k}})_{k\in J}$. Moreover,
\eqref{eq:affine representation} holds for some functions $\phi(t,u)$
and $\psi(t,u)$ that are uniquely determined by the generalized Riccati
differential equations: for each $u=(v,w)\in\mathbb{C}_{\leqslant0}^{m}\times\mathrm{i}\mathbb{R}^{n}$,
\begin{align}
\partial_{t}\phi(t,u) & =F\left(\psi(t,u)\right),\quad\phi(0,u)=0,\nonumber \\
\partial_{t}\psi^{I}(t,u) & =R^{I}\left(\psi^{I}\left(t,u\right),\mathrm{e}^{\beta_{JJ}^{\top}t}w\right),\quad\psi^{I}\left(0,u\right)=v\label{eq:riccati equation for psi ell}\\
\psi^{J}(t,u) & =\mathrm{e}^{\beta_{JJ}^{\top}t}w,\label{eq:riccati equation for psi J}
\end{align}
where
\begin{align}
F(u) & =\langle u,au\rangle+\langle b,u\rangle+\int_{D\backslash\{0\}}\left(\mathrm{e}^{\langle u,\xi\rangle}-1-\langle u_{J},\xi_{J}\rangle\mathbbm{1}_{\lbrace\Vert\xi\Vert\leqslant1\rbrace}\left(\xi\right)\right)m\left(\mathrm{d}\xi\right)\label{eq:representation of F(u)}
\end{align}
and $R^{I}=(R_{1},\ldots,R_{m})$ with
\[
R_{i}(u)=\langle u,\alpha_{i}u\rangle+\sum_{k=1}^{d}\beta_{ki}u_{k}+\int_{D\backslash\{0\}}\left(\mathrm{e}^{\langle u,\xi\rangle}-1-\langle u,\xi\rangle\right)\mu_{i}\left(\mathrm{d}\xi\right),\quad i\in I.
\]
\end{thm}

\begin{rem}
If an affine process $X$ with state space $D$ and a set of admissible
parameters $(a,\alpha,b,\beta,m,\mu)$ satisfy a relation as in Theorem
\ref{thm:characterization of affine processes}, then we say that
$X$ is an affine process with admissible parameters $(a,\alpha,b,\beta,m,\mu)$.
\end{rem}

The following lemma is a consequence of the condition (iv) in Definition
\ref{def:admissible parameters}.
\begin{lem}
\label{lem:space-differentiability of F R phi psi}Let $X$ be an
affine process with state space $D$ and admissible parameters $(a,\alpha,b,\beta,m,\mu)$.
Let $R$ and $\psi$ be as in Theorem \ref{thm:characterization of affine processes}.
For each $i\in I$ it holds that $R_{i}\in C^{1}(\mathcal{U})$ and
$\psi_{i}\in C^{1}(\mathbb{R}_{\geqslant0}\times\mathcal{U})$.
\end{lem}

To see that Lemma \ref{lem:space-differentiability of F R phi psi}
is true, we only need to apply Lemmas 5.3 and 6.5 of \cite{MR1994043}.

\subsection{Main result}

Our main result of this paper is the following.
\begin{thm}
\label{thm:stationarity of affine processes}Let $X$ be an affine
process with state space $\mathbb{R}_{\geqslant0}^{m}\times\mathbb{R}^{n}$
and admissible parameters $(a,\alpha,b,\beta,m,\mu)$ in the sense
of Definition \ref{def:admissible parameters}. If
\[
\beta\in\mathbb{M}_{d}^{-}\quad\text{and}\quad\int_{\left\{ \left\Vert \xi\right\Vert >1\right\} }\log\left\Vert \xi\right\Vert m\left(\mathrm{d}\xi\right)<\infty,
\]
then the law of $X_{t}$ converges weakly to a limiting distribution
$\pi$, which is independent of $X_{0}$ and whose characteristic
function is given by
\[
\int_{D}\mathrm{e}^{\langle u,x\rangle}\pi\left(\mathrm{d}x\right)=\exp\left\{ \int_{0}^{\infty}F\left(\psi(s,u)\right)\mathrm{d}s\right\} ,\quad\text{ }u\in\mathcal{U}.
\]
Moreover, the limiting distribution $\pi$ is the unique stationary
distribution for $X$.
\end{thm}

\begin{rem}
\label{rem: beta condition} In virtue of the definition of admissible
parameters, we can write $\beta\in\mathbb{R}^{d\times d}$ in the
following way:

\renewcommand*\arraystretch{1.4}\begin{equation}\label{eq: form of beta}
\beta=
\left( \begin{array}{c;{2pt/2pt}c}
\beta_{II} & 0 \\ \hdashline[2pt/2pt]
\beta_{JI} & \beta_{JJ}
\end{array}\right),
\end{equation}where $\beta_{II}\in\mathbb{R}^{m\times m}$, $\beta_{JI}\in\mathbb{R}^{n\times m}$
and $\beta_{JJ}\in\mathbb{R}^{n\times n}$. It is easy to see that
$\beta\in\mathbb{M}_{d}^{-}$ is equivalent to the fact that $\beta_{II}\in\mathbb{M}_{m}^{-}$
and $\beta_{JJ}\in\mathbb{M}_{n}^{-}$.\\
\end{rem}

We now make a few comments on Theorem \ref{thm:stationarity of affine processes}.
To our knowledge, Theorem \ref{thm:stationarity of affine processes}
seems to be the first result towards the existence of limiting distributions
for affine processes on $D$ in such a generality. It includes many
previous results as special cases. In particular, it covers \cite[Theorem 2.4]{MR2599675}
for affine diffusions, and partially extends \cite[Theorem 4.1]{MR738769}
for OU-type processes and \cite[Corollary 2]{MR0295450} for $1$-dimensional
CBI processes.\textcolor{red}{{} }However, we are not able to show $\int_{\left\{ \left\Vert \xi\right\Vert >1\right\} }\log\left\Vert \xi\right\Vert m\left(\mathrm{d}\xi\right)<\infty$,
provided that $\beta\in\mathbb{M}_{d}^{-}$ and the stationarity of
$X$ is known. \\

Our strategy of proving Theorem \ref{thm:stationarity of affine processes}
is as follows. Clearly, to prove the weak convergence of the distribution
of $X_{t}$ to $\pi$, it is essential to establish the pointwise
convergence of the corresponding characteristic functions, i.e.,
\[
\mathbb{E}_{x}\left[\mathrm{e}^{\langle X_{t},u\rangle}\right]=\exp\left\{ \phi(t,u)+\langle x,\psi(t,u)\rangle\right\} \to\exp\left\{ \int_{0}^{\infty}F(\psi(s,u))\mathrm{d}s\right\} \quad\text{as }t\to\infty.
\]
We will proceed in two steps. In the first step, we prove that for
each $u\in\mathcal{U}$, $\psi(t,u)$ converges to zero exponentially
fast. For $u$ in a small neighborhood of the origin, this convergence
follows by a fine analysis of the generalized Riccati equations \eqref{eq:riccati equation for psi ell},
\eqref{eq:representation of F(u)} and an application of the linearized
stability theorem for ODEs. Then, by some probabilistic arguments,
we show that $\psi(t,u)$ reaches every neighborhood of the origin
for large enough $t$. The essential observation here is the tightness
of the laws of $X_{t}$, $t\geqslant0$. This is a simple consequence
of the uniform boundedness for the first moment of $X_{t}$, $t\geqslant0$,
which we show in Proposition \ref{thm:uniform boundedness for the first moment of X_t}.
We thus obtain the desired convergence speed of $\psi(t,u)\to0$ by
the semi-flow property \eqref{eq:semi flow property}. In the second
step, we show that
\begin{equation}
\phi(t,u)=\int_{0}^{t}F(\psi(s,u))\mathrm{d}s\to\int_{0}^{\infty}F(\psi(s,u))\mathrm{d}s\quad\text{as }t\to\infty.\label{eq: conv. phi}
\end{equation}
Since $\psi(s,u)\to0$ exponentially fast as $s\to\infty$, we will
see that the convergence in \eqref{eq: conv. phi} is naturally connected
with the condition $\int_{\left\{ \left\Vert \xi\right\Vert >1\right\} }\log\left\Vert \xi\right\Vert m\left(\mathrm{d}\xi\right)<\infty$.
Finally, the stationarity of $\pi$ can be derived using the semi-flow
property.

\section{Large time behavior of the function $\psi(t,u)$ }

In this section we consider an affine process $X$ with admissible
parameters $(a,\alpha,b,\beta,m,\mu)$ and assume that
\begin{equation}
a=0,\ b=0,\ m=0.\label{eq: assum sect. 3}
\end{equation}
In particular, we have $F\equiv0$ as well as $\phi\equiv0$. We will
show that if $\beta\in\mathbb{M}_{d}^{-}$, then $\psi(t,u)\to0$
exponentially fast as $t\to\infty$.
\begin{rem}
\label{rem:The-assumption-that}The assumption that $a=0,\ b=0$ and
$m=0$ is not essential. Indeed, Proposition \ref{thm:exponential convergence of psi for all u},
as the main result of this section, remains true if we drop  Assumption
(\ref{eq: assum sect. 3}). This follows from the following observation:
when we study the properties of the function $\psi(t,u)$, the parameters
$a,\thinspace b$ and $m$ do not play a role.
\end{rem}

\subsection{Uniform boundedness for the first moment of $X_{t}$, $t\geqslant0$ }

The aim we pursue in this subsection is to establish the uniform boundedness
for the first moment of $X_{t}$, $t\geqslant0$. We start with some
approximations of $X$, which were introduced in \cite{MR3540486}.

For $K\in(1,\infty)$, let
\[
\mu_{K,i}(\mathrm{d}\xi):=\mathbbm{1}_{\{\left\Vert \xi\right\Vert \leqslant K\}}(\xi)\mu_{i}(\mathrm{d}\xi),
\]
and denote by $(X_{K,t})_{t\geqslant0}$ the affine process with admissible
parameters $(a=0,\alpha,b=0,\beta,m=0,\mu_{K})$, where $\mu_{K}=(\mu_{K,1},\ldots,\mu_{K,m})$.
Then we have
\[
\mathbb{E}_{x}\left[\mathrm{e}^{\langle X_{K,t},u\rangle}\right]=\exp\left\{ \langle x,\psi_{K}\left(t,u\right)\rangle\right\} ,\quad t\geqslant0,\;x\in D,\ u\in\mathcal{U},
\]
for some function $\psi_{K}:\mathbb{R}_{\geqslant0}\times\mathcal{U}\to\mathbb{C}^{d}$.
By \eqref{eq:riccati equation for psi ell} and \eqref{eq:riccati equation for psi J},
we know that $\psi_{K}=(\psi_{K}^{I},\psi^{J})$, where $\psi^{J}(t,u)=\exp(\beta_{JJ}^{\top}t)w$
for $u=(v,w)\in\mathbb{C}_{\leqslant0}^{m}\times\mathrm{i}\mathbb{R}^{n}$
and $\psi_{K}^{I}$ satisfies the generalized Riccati equation
\[
\partial_{t}\psi_{K}^{I}\left(t,u\right)=R_{K}^{I}\left(\psi_{K}^{I}\left(t,u\right),\mathrm{e}^{\beta_{JJ}^{\top}t}w\right),\quad\psi_{K}^{I}(0,u)=v\in\mathbb{C}_{\leqslant0}^{m},
\]
where $R_{K}^{I}=(R_{K,i},\ldots,R_{K,m})^{\top}$ with
\[
R_{K,i}(u)=\langle u,\alpha_{i}u\rangle+\sum_{k=1}^{d}\beta_{ki}u_{k}+\int_{D\backslash\{0\}}\left(\mathrm{e}^{\langle u,\xi\rangle}-1-\langle u,\xi\rangle\right)\mu_{K,i}\left(\mathrm{d}\xi\right),\quad i\in I.
\]

\begin{lem}
\label{lem:approximation lemma} F\textup{or each $t\in\mathbb{R}_{\geqslant0}$
and $u\in\mathcal{U}$, $\psi_{K}(t,u)$ converges to $\psi(t,u)$
as $K\to\infty$.}
\end{lem}

\begin{proof}
Clearly, we only need to show the pointwise convergence of $\psi_{K}^{I}$
to $\psi^{I}$. Let $u=(v,w)\in\mathbb{C}_{\leqslant0}^{m}\times\mathrm{i}\mathbb{R}^{n}$
and $T>0$ be fixed.

By the Riccati equations for $\psi^{I}$ and $\psi_{K}^{I}$, we get
\begin{equation}
\psi^{I}(t,u)=v+\int_{0}^{t}R^{I}\left(\psi^{I}\left(s,u\right),\mathrm{e}^{\beta_{JJ}^{\top}s}w\right)\mathrm{d}s,\quad t\geqslant0,\label{eq: Lemma 3.1 psi I}
\end{equation}
and
\begin{equation}
\psi_{K}^{I}(t,u)=v+\int_{0}^{t}R_{K}^{I}\left(\psi_{K}^{I}\left(s,u\right),\mathrm{e}^{\beta_{JJ}^{\top}s}w\right)\mathrm{d}s,\quad t\geqslant0.\label{eq: Lemma 3.1 psi K}
\end{equation}
In view of the formula (6.16) in the proof of \cite[Propostion 6.1]{MR1994043},
we have
\begin{align}
\sup_{t\in[0,T]}\left\Vert \psi_{K}^{I}(t,u)\right\Vert ^{2} & \leqslant\sup_{t\in[0,T]}\left(\left\Vert v\right\Vert ^{2}+c_{1}\int_{0}^{t}\left(1+\left\Vert \mathrm{e}^{\beta_{JJ}^{\top}s}w\right\Vert ^{2}\right)\mathrm{d}s\right)\nonumber \\
 & \quad\quad\quad\quad\quad\quad\quad\times\exp\left\{ c_{1}\int_{0}^{t}\left(1+\left\Vert \mathrm{e}^{\beta_{JJ}^{\top}s}w\right\Vert ^{2}\right)\mathrm{d}s\right\} \nonumber \\
 & \leqslant\left(\left\Vert v\right\Vert ^{2}+c_{1}\int_{0}^{T}\left(1+\left\Vert \mathrm{e}^{\beta_{JJ}^{\top}s}w\right\Vert ^{2}\right)\mathrm{d}s\right)\nonumber \\
 & \quad\quad\quad\quad\times\exp\left\{ c_{1}\int_{0}^{T}\left(1+\left\Vert \mathrm{e}^{\beta_{JJ}^{\top}s}w\right\Vert ^{2}\right)\mathrm{d}s\right\} ,\label{eq:psi_K estimate}
\end{align}
for some positive constant $c_{1}$. Moreover, by checking carefully
the proof of \cite[Propostion 6.1]{MR1994043} and noting that $\mu_{K,i}\leqslant\mu_{i}$,
we can actually choose $c_{1}$ in such a way that it depends only
on the parameters $\alpha,\thinspace\beta,\thinspace\mu$. So $c_{1}$
is independent of $K$. Similarly, the same inequality holds for $\psi^{I}$:
\begin{align*}
\sup_{t\in[0,T]}\left\Vert \psi^{I}(t,u)\right\Vert ^{2} & \leqslant\left(\left\Vert v\right\Vert ^{2}+c_{1}\int_{0}^{T}\left(1+\left\Vert \mathrm{e}^{\beta_{JJ}^{\top}s}w\right\Vert ^{2}\right)\mathrm{d}s\right)\\
 & \quad\quad\quad\quad\times\exp\left\{ c_{1}\int_{0}^{T}\left(1+\left\Vert \mathrm{e}^{\beta_{JJ}^{\top}s}w\right\Vert ^{2}\right)\mathrm{d}s\right\} .
\end{align*}

According to Lemma \ref{lem:space-differentiability of F R phi psi},
the mapping $u\mapsto R^{I}(u):\mathcal{U}\to\mathbb{C}^{m}$ is locally
Lipschitz continuous. Therefore, for each $L>0$, there exists a constant
$c_{2}=c_{2}(L)>0$ such that
\begin{equation}
\Vert R_{i}(u_{1})-R_{i}(u_{2})\Vert\leqslant c_{2}\left\Vert u_{1}-u_{2}\right\Vert ,\quad\text{for all \ensuremath{i\in I} and }\left\Vert u_{1}\right\Vert ,\thinspace\left\Vert u_{2}\right\Vert \leqslant L.\label{eq:lipschitz continuity of R}
\end{equation}
In addition, it is easy to see that for $u\in\mathcal{U}$,
\begin{align}
\Vert R_{i}(u)-R_{K,i}(u)\Vert & =\left|\int_{\left\{ \Vert\xi\Vert>K\right\} }\left(\mathrm{e}^{\langle u,\xi\rangle}-1-\langle u,\xi\rangle\right)\mu_{i}\left(\mathrm{d}\xi\right)\right|\nonumber \\
 & \leqslant\int_{\left\{ \Vert\xi\Vert>K\right\} }2\mu_{i}\left(\mathrm{d}\xi\right)+\Vert u\Vert\int_{\left\{ \Vert\xi\Vert>K\right\} }\left\Vert \xi\right\Vert \mu_{i}\left(\mathrm{d}\xi\right)\nonumber \\
 & \leqslant\varepsilon_{K}\left(1+\Vert u\Vert\right),\label{eq:supremum of R_1 and R_K,1}
\end{align}
where $\varepsilon_{K}:=\sum_{i=1}^{m}\int_{\{\Vert\xi\Vert>K\}}\left(2+\Vert\xi\Vert\right)\mu_{i}(\mathrm{d\xi)}$.
Note that $\varepsilon_{K}\to0$ as $K\to\infty$ by dominated convergence.

Let
\[
g_{K}(t):=\left\Vert \psi^{I}(t,u)-\psi_{K}^{I}(t,u)\right\Vert ,\quad t\in\left[0,T\right].
\]
By \eqref{eq: Lemma 3.1 psi I} and \eqref{eq: Lemma 3.1 psi K},
we have
\begin{align}
g_{K}(t) & \leqslant\left\Vert \int_{0}^{t}R^{I}\left(\psi^{I}\left(s,u\right),\mathrm{e}^{\beta_{JJ}^{\top}s}w\right)\mathrm{d}s-\int_{0}^{t}R_{K}^{I}\left(\psi_{K}^{I}\left(s,u\right),\mathrm{e}^{\beta_{JJ}^{\top}s}w\right)\mathrm{d}s\right\Vert \nonumber \\
 & \leqslant\sum_{i=1}^{m}\int_{0}^{t}\left\Vert R_{i}\left(\psi^{I}\left(s,u\right),\mathrm{e}^{\beta_{JJ}^{\top}s}w\right)-R_{i}\left(\psi_{K}^{I}\left(s,u\right),\mathrm{e}^{\beta_{JJ}^{\top}s}w\right)\right\Vert \mathrm{d}s\nonumber \\
 & \quad+\sum_{i=1}^{m}\int_{0}^{t}\left\Vert R_{i}\left(\psi_{K}^{I}\left(s,u\right),\mathrm{e}^{\beta_{JJ}^{\top}s}w\right)-R_{K,i}\left(\psi_{K}^{I}\left(s,u\right),\mathrm{e}^{\beta_{JJ}^{\top}s}w\right)\right\Vert \mathrm{d}s.\label{eq: lemma 3.1 first esti g_K}
\end{align}
In virtue of \eqref{eq:psi_K estimate}, there exists a constant $c_{3}=c_{3}(T)>0$
such that
\[
\sup_{K\in[1,\infty)}\sup_{s\in\left[0,T\right]}\left\Vert \psi_{K}^{I}\left(s,u\right)\right\Vert \leqslant c_{3}<\infty,
\]
which implies
\begin{equation}
\sup_{K\in[1,\infty)}\sup_{s\in\left[0,T\right]}\left\Vert \left(\psi_{K}^{I}(s,u),\mathrm{e}^{\beta_{JJ}^{\top}s}w\right)\right\Vert \leqslant c_{4}<\infty.\label{eq: lemma 3.1 sup psi K}
\end{equation}
So, for $0<s\leqslant T$, we get

\begin{equation}
\left\Vert R_{i}\left(\psi^{I}\left(s,u\right),\mathrm{e}^{\beta_{JJ}^{\top}s}w\right)-R_{i}\left(\psi_{K}^{I}\left(s,u\right),\mathrm{e}^{\beta_{JJ}^{\top}s}w\right)\right\Vert \leqslant c_{5}\left\Vert \psi^{I}\left(s,u\right)-\psi_{K}^{I}\left(s,u\right)\right\Vert \label{eq: lemma 3.1 second esti differ}
\end{equation}
from \eqref{eq:lipschitz continuity of R}, and obtain
\begin{equation}
\left\Vert R_{i}\left(\psi_{K}^{I}\left(s,u\right),\mathrm{e}^{\beta_{JJ}^{\top}s}w\right)-R_{K,i}\left(\psi_{K}^{I}\left(s,u\right),\mathrm{e}^{\beta_{JJ}^{\top}s}w\right)\right\Vert \leqslant\varepsilon_{K}\left(1+c_{6}\right)\label{eq:estimate for Ri(psi_K)-R_K(psi_K)}
\end{equation}
from \eqref{eq:supremum of R_1 and R_K,1} and \eqref{eq: lemma 3.1 sup psi K}.
Here, $c_{5},\thinspace c_{6}>0$ are constants not depending on $K$.

Combining \eqref{eq: lemma 3.1 first esti g_K}, (\ref{eq: lemma 3.1 second esti differ})
and \eqref{eq:estimate for Ri(psi_K)-R_K(psi_K)} yields, for $t\in\left[0,T\right]$,
\begin{align*}
g_{K}(t) & \leqslant c_{5}m\int_{0}^{t}\left\Vert \psi^{I}\left(s,u\right)-\psi_{K}^{I}\left(s,u\right)\right\Vert \mathrm{d}s+m\varepsilon_{K}\left(1+c_{6}\right)t\\
 & =c_{5}m\int_{0}^{t}g_{K}(s)\mathrm{d}s+m\varepsilon_{K}\left(1+c_{6}\right)t.
\end{align*}
Gronwall's inequality implies
\begin{align*}
g_{K}(t) & \leqslant m\varepsilon_{K}\left(1+c_{6}\right)t+m^{2}\varepsilon_{K}\left(1+c_{6}\right)c_{5}\int_{0}^{t}s\mathrm{e}{}^{c_{5}m\left(t-s\right)}\mathrm{d}s\\
 & \leqslant m\varepsilon_{K}\left(1+c_{6}\right)\left(T+c_{5}mT^{2}\mathrm{e}^{c_{5}mT}\right),\qquad t\in[0,T].
\end{align*}
Since $\varepsilon_{K}\to0$ as $K\to\infty$, we see that $g_{K}(t)\to0$
and thus
\[
\psi_{K}^{I}\left(t,u\right)\to\psi^{I}\left(t,u\right),\quad\text{for all }t\in\left[0,T\right].
\]
\end{proof}
For $K\in(1,\infty)$, the generator $\mathcal{A}_{K}$ of $(X_{K,t})_{t\geqslant0}$
is given by
\begin{align*}
\mathcal{A}_{K}f(x) & =\sum_{k,l=1}^{d}\left(\sum_{i=1}^{m}\alpha_{i,kl}x_{i}\right)\frac{\partial^{2}f(x)}{\partial x_{k}\partial x_{l}}+\langle\beta x,\nabla f(x)\rangle\\
 & \quad+\sum_{i=1}^{m}x_{i}\int_{D\backslash\{0\}}\left(f\left(x+\xi\right)-f\left(x\right)-\langle\nabla f(x),\xi\rangle\right)\mu_{K,i}\left(\mathrm{d}\xi\right),
\end{align*}
defined for every $f\in C_{c}^{2}(D)$.

To avoid the complication of discussing the domain of definition for
the generator $\mathcal{A}_{K}$, we introduce the operator $\mathcal{A}_{K}^{\sharp}$,
which was also used in \cite{MR1994043}.
\begin{defn}
If $f\in C^{2}(D)$ is such that for all $x\in D$,
\[
\sum_{i=1}^{m}\int_{D\backslash\{0\}}\left|f\left(x+\xi\right)-f\left(x\right)-\langle\nabla f(x),\xi\rangle\right|\mu_{K,i}\left(\mathrm{d}\xi\right)<\infty,
\]
then we say that $\mathcal{A}_{K}^{\sharp}f$ is well-defined and
let
\begin{align*}
\mathcal{A}_{K}^{\sharp}f(x) & :=\sum_{k,l=1}^{d}\left(\sum_{i=1}^{m}\alpha_{i,kl}x_{i}\right)\frac{\partial^{2}f(x)}{\partial x_{k}\partial x_{l}}+\langle\beta x,\nabla f(x)\rangle\\
 & \quad+\sum_{i=1}^{m}x_{i}\int_{D\backslash\{0\}}\left(f\left(x+\xi\right)-f\left(x\right)-\langle\nabla f(x),\xi\rangle\right)\mu_{K,i}\left(\mathrm{d}\xi\right)
\end{align*}
for $x\in D$.
\end{defn}

It is easy to see that if $f\in C^{2}(D)$ has bounded first and second
order derivatives, then $\mathcal{A}_{K}^{\sharp}f$ is well-defined.

Recall that the matrix $\beta$ can be written as in (\ref{eq: form of beta}).
We define the following matrices
\[
M_{1}:=\int_{0}^{\infty}\mathrm{e}^{t\beta_{II}^{\top}}\mathrm{e}^{t\beta_{II}}\mathrm{d}t\quad\text{and}\quad M_{2}:=\int_{0}^{\infty}\mathrm{e}^{t\beta_{JJ}^{\top}}\mathrm{e}^{t\beta_{JJ}}\mathrm{d}t.
\]
Since $\beta_{II}\in\mathbb{M}_{m}^{-}$ and $\beta_{JJ}\in\mathbb{M}_{n}^{-}$,
the matrices $M_{1}$ and $M_{2}$ are well-defined. Moreover, we
have that $M_{1}\in\mathbb{S}_{m}^{++}$ and $M_{2}\in\mathbb{S}_{n}^{++}$.
In the following we will often write $x=(y,z)\in\mathbb{R}_{\geqslant0}^{m}\times\mathbb{R}^{n}$\textcolor{red}{{}
}for\textcolor{red}{{} }$x\in D$. For $y_{1},\thinspace y_{2}\in\mathbb{R}_{\geqslant0}^{m}$
and $z_{1},\thinspace z_{2}\in\mathbb{R}^{n}$, we define
\[
\langle y_{1},y_{2}\rangle_{I}:=\int_{0}^{\infty}\langle\mathrm{e}^{t\beta_{II}}y_{1},\mathrm{e}^{t\beta_{II}}y_{2}\rangle\mathrm{d}t\quad\text{and}\quad\langle z_{1},z_{2}\rangle_{J}:=\int_{0}^{\infty}\langle\mathrm{e}^{t\beta_{JJ}}z_{1},\mathrm{e}^{t\beta_{JJ}}z_{2}\rangle\mathrm{d}t.
\]
It is easily verified that $\langle\cdot,\cdot\rangle_{I}$ and $\langle\cdot,\cdot\rangle_{J}$
define inner products on $\mathbb{R}^{m}$ and $\mathbb{R}^{n}$,
respectively. Moreover, we have that
\[
\langle y_{1},y_{2}\rangle_{I}=y_{2}^{\top}M_{1}y_{1}=\langle y_{1},M_{1}y_{2}\rangle\quad\text{and}\quad\langle z_{1},z_{2}\rangle_{J}=z_{2}^{\top}M_{2}z_{1}=\langle z_{1},M_{2}z_{2}\rangle.
\]
The norms on $\mathbb{R}^{m}$ and $\mathbb{R}^{n}$ induced by the
scalar products $\langle\cdot,\cdot\rangle_{I}$ and $\langle\cdot,\cdot\rangle_{J}$
are denoted by
\[
\Vert y\Vert_{I}:=\sqrt{\langle y,y\rangle_{I}}\quad\text{and}\quad\Vert z\Vert_{J}:=\sqrt{\langle z,z\rangle_{J}},
\]
respectively.\\

In the following lemma we construct a Lyapunov function $V$ for $(X_{K,t})_{t\geqslant0}$.
Note that the definition of $V$ does not depend on $K$.
\begin{lem}
\label{lem:foster-lyapunov estimate lemma}Assume $m\geqslant1$ and
$n\geqslant1$. Suppose that $\beta\in\mathbb{M}_{d}^{-}$. Let $V\in C^{2}(D,\mathbb{R})$
be such that $V>0$ on $D$ and
\[
V(x)=\left(\langle y,y\rangle_{I}+\varepsilon\langle z,z\rangle_{J}\right)^{1/2},\quad\text{whenever \ensuremath{x=(y,z)\in\mathbb{R}_{\geqslant0}^{m}\times\mathbb{R}^{n}} with }\Vert x\Vert>2.
\]
Here $\varepsilon>0$ is some small enough constant. Then $\mathcal{A}_{K}^{\sharp}V$
is well-defined and $V$ is a Lyapunov function for $(X_{K,t})_{t\geqslant0}$,
that is, there exist positive constants $c$ and $C$ such that
\[
\mathcal{A}_{K}^{\sharp}V(x)\leqslant-cV(x)+C,\quad\mbox{for all \ensuremath{x\in D.}}
\]
Moreover, the constants $c$ and $C$ can be chosen to be independent
of $K$.
\end{lem}

\begin{proof}
For $x_{1}=(y_{1},z_{1})\in\mathbb{R}_{\geqslant0}^{m}\times\mathbb{R}^{n}$
and $x_{2}=(y_{2},z_{2})\in\mathbb{R}_{\geqslant0}^{m}\times\mathbb{R}^{n}$,
we define
\[
\langle x_{1},x_{2}\rangle_{\beta}:=\langle y_{1},z_{1}\rangle_{I}+\varepsilon\langle y_{2},z_{2}\rangle_{J},
\]
where $\varepsilon>0$ is a small constant to be determined later.
Set $\tilde{V}(x):=\left(\langle x,x\rangle_{\beta}\right)^{1/2}$,
$x\in D$. Then $\widetilde{V}$ is smooth on $\left\{ x\in D:\|x\|>1\right\} $.
By the extension lemma for smooth functions (see \cite[Lemma 2.26]{MR2954043}),
we can easily find a function $V\in C^{\infty}(D,\mathbb{R})$ such
that $V>0$ on $D$ and $V(x)=\widetilde{V}(x)=\left(\langle x,x\rangle_{\beta}\right)^{1/2}$
for $\|x\|>2$. So for all $x=(y,z)\in\mathbb{R}_{\geqslant0}^{m}\times\mathbb{R}^{n}$
with $\Vert x\Vert>2$, we have
\begin{equation}
\nabla V(y,z)=V(y,z)^{-1}\begin{pmatrix}M_{1}y\\
\varepsilon M_{2}z
\end{pmatrix}\label{eq:gradient of V}
\end{equation}
and\renewcommand*\arraystretch{1.7}\begin{equation}\label{eq:second derivative of V}
\nabla^2 V(y,z)=
\left( \begin{array}{cc}
\frac{M_{1}}{V(y,z)}-\frac{\left(M_{1}y\right)\left(M_{1}y\right)^{\top}}{V(y,z)^{3}}
& \frac{-\varepsilon\left(M_{1}y\right)\left(M_{2}z\right)^{\top}}{V(y,z)^{3}} \\
\frac{-\varepsilon\left(M_{1}y\right)\left(M_{2}z\right)^{\top}}{V(y,z)^{3}}
& \frac{\varepsilon M_{2}}{V(y,z)}-\frac{\varepsilon^{2}\left(M_{2}z\right)\left(M_{2}z\right)^{\top}}{V(y,z)^{3}}
\end{array} \right).
\end{equation}We write $\mathcal{A}_{K}^{\sharp}V=\mathcal{D}V+\mathcal{J}_{K}V$,
where
\begin{align}
\mathcal{D}V(x) & \text{:=}\sum_{k,l=1}^{d}\langle\alpha_{I,kl},x_{I}\rangle\frac{\partial^{2}V(x)}{\partial x_{k}\partial x_{l}}+\langle\beta x,\nabla V(x)\rangle,\label{eq: defi DV}\\
\mathcal{J}_{K}V(x) & :=\sum_{i=1}^{m}x_{i}\int_{D\backslash\{0\}}\left(V\left(x+\xi\right)-V\left(x\right)-\langle\nabla V(x),\xi\rangle\right)\mu_{K,i}\left(\mathrm{d}\xi\right).\label{eq: defin J_KV}
\end{align}
We now estimate $\mathcal{D}V(x)$ and $\mathcal{J}_{K}V(x)$ separately.
Let us first consider $\mathcal{D}V(x)$. We may further split $\mathcal{D}V(x)$
into the drift part and the diffusion part.

\textit{Drift.} Recall that $\beta_{IJ}=0$. Consider $x=(y,z)$ with
$\|x\|>2$. It follows from \eqref{eq:gradient of V} that
\begin{align*}
\langle\beta x,\nabla V(x)\rangle & =\langle\begin{pmatrix}\beta_{II}y\\
\beta_{JI}y+\beta_{JJ}z
\end{pmatrix},\begin{pmatrix}V(y,z)^{-1}M_{1}y\\
V(y,z)^{-1}\varepsilon M_{2}z
\end{pmatrix}\rangle\\
 & =V(y,z)^{-1}\left(\langle\beta_{II}y,M_{1}y\rangle+\langle\beta_{JI}y,\varepsilon M_{2}z\rangle+\langle\beta_{JJ}z,\varepsilon M_{2}z\rangle\right).
\end{align*}
The first and the third inner product on the right-hand side may be
estimated similarly. Namely, we have
\[
V(y,z)^{-1}\langle\beta_{II}y,M_{1}y\rangle=\frac{1}{2}V(y,z)^{-1}y^{\top}\left(M_{1}\beta_{II}+\beta_{II}^{\top}M_{1}\right)y.
\]
The definition of $M_{1}$ implies
\begin{align*}
M_{1}\beta_{II}+\beta_{II}^{\top}M_{1} & =\int_{0}^{\infty}\left(\mathrm{e}^{t\beta_{II}^{\top}}\mathrm{e}^{t\beta_{II}}\beta_{II}+\beta_{II}^{\top}\mathrm{e}^{t\beta_{II}^{\top}}\mathrm{e}^{t\beta_{II}}\right)\mathrm{d}t\\
 & =\int_{0}^{\infty}\left(\frac{\mathrm{d}}{\mathrm{d}t}\mathrm{e}^{t\beta_{II}^{\top}}\mathrm{e}^{t\beta_{II}}\right)\mathrm{d}t\\
 & =\left.\mathrm{e}^{t\beta_{II}^{\top}}\mathrm{e}^{t\beta_{II}}\right|_{t=0}^{\infty}\\
 & =-I_{m},
\end{align*}
where $I_{m}$ denotes the $m\times m$ identity matrix. Hence
\[
V(y,z)^{-1}\langle\beta_{II}y,M_{1}y\rangle=-\frac{1}{2}V(y,z)^{-1}y^{\top}y.
\]
Since all norms on $\mathbb{R}^{m}$ are equivalent, we have
\[
-y^{\top}y\leqslant-c_{1}y^{\top}M_{1}y=-c_{1}\langle y,y\rangle_{I}\leqslant-c_{1}\Vert y\Vert_{I}^{2},
\]
for some positive constant $c_{1}$ that is independent of $K$. So
\begin{align}
V(y,z)^{-1}\langle\beta_{II}y,M_{1}y\rangle & \leqslant-c_{1}\Vert y\Vert_{I}^{2}V(y,z)^{-1}.\label{eq:drift estimate ell ell}
\end{align}
In the very same way we obtain
\begin{equation}
V(y,z)^{-1}\langle\beta_{JJ}z,\varepsilon M_{2}z\rangle\leqslant-c_{2}\varepsilon\Vert z\Vert_{J}^{2}V(y,z)^{-1},\label{eq:drift estimate JJ}
\end{equation}
for some constant $c_{2}>0$. To estimate the remaining term, we can
use Cauchy Schwarz inequality to obtain
\begin{align*}
\left|V(y,z)^{-1}\langle\beta_{JI}y,\varepsilon M_{2}z\rangle\right| & \leqslant\varepsilon V(y,z)^{-1}\left\Vert \beta_{JI}y\right\Vert \Vert M_{2}z\Vert\\
 & \leqslant c_{3}\varepsilon V(y,z)^{-1}\left\Vert y\right\Vert \Vert z\Vert,
\end{align*}
for some constant $c_{3}>0$. Using the fact that all norms on $\mathbb{R}^{d}$
are equivalent, we get
\begin{align}
\left|V(y,z)^{-1}\langle\beta_{JI}y,\varepsilon M_{2}z\rangle\right| & \leqslant\varepsilon c_{4}V(y,z)^{-1}\Vert y\Vert_{I}\Vert z\Vert_{J}\nonumber \\
 & =c_{4}\frac{\sqrt{\varepsilon}\sqrt{\langle y,y\rangle_{I}}\sqrt{\varepsilon\langle z,z\rangle_{J}}}{\sqrt{\langle y,y\rangle_{I}+\varepsilon\langle z,z\rangle_{J}}}\nonumber \\
 & \leqslant c_{4}\sqrt{\varepsilon}\Vert y\Vert_{I}.\label{eq:drift estimate ell J}
\end{align}
Combining \eqref{eq:drift estimate ell ell}, \eqref{eq:drift estimate JJ}
and \eqref{eq:drift estimate ell J}, we obtain
\begin{align*}
\langle\beta x,\nabla V(x)\rangle & \leqslant-c_{1}\Vert y\Vert_{I}^{2}V(y,z)^{-1}-\varepsilon c_{2}\Vert z\Vert_{J}^{2}V(y,z)^{-1}+c_{4}\sqrt{\varepsilon}\Vert y\Vert_{I}\\
 & \leqslant-c_{5}\left(\langle y,y\rangle_{I}+\varepsilon\langle z,z\rangle_{J}\right)V(y,z)^{-1}+c_{4}\sqrt{\varepsilon}\Vert y\Vert_{I}\\
 & \leqslant-c_{5}V(y,z)+c_{4}\sqrt{\varepsilon}V(y,z),
\end{align*}
where $c_{5}:=c_{1}\wedge c_{2}>0$. Since $c_{4}$ and $c_{5}$ depend
only on $\beta$ but not on $\varepsilon$, by choosing $\varepsilon=\varepsilon_{0}>0$
sufficiently small, we get
\begin{equation}
\langle\beta x,\nabla V(x)\rangle\leqslant-c_{6}V(x),\quad x\in D\text{\ensuremath{\quad\mbox{with \ensuremath{\|x\|>2}}}}.\label{eq:estimate drift}
\end{equation}
\textit{\emph{From}}\textit{ }\textit{\emph{now on we take $\varepsilon=\varepsilon_{0}$
as fixed. In particular, the upcoming constants $c_{7}-c_{11}$ may
depend on $\varepsilon$.}}

\textit{Diffusion.} By \eqref{eq:second derivative of V}, we have
\begin{equation}
\left|\frac{\partial^{2}V(x)}{\partial x_{k}\partial x_{l}}\right|\leqslant\frac{c_{7}}{V(x)},\quad\text{for all }\|x\|>2,\ k,l\in\{1,\ldots,d\},\label{eq: exact bound for nabla2}
\end{equation}
where $c_{7}>0$ is a constant. This implies
\[
\sup_{x\in D}\left|x_{i}\frac{\partial^{2}V(x)}{\partial x_{k}\partial x_{l}}\right|<\infty,\quad\text{for all }i\in I\ \text{and\ }k,l\in\{1,\ldots,d\}.
\]
We conclude that
\begin{equation}
\left|\sum_{k,l=1}^{d}\left(\sum_{i\in I}\alpha_{i,kl}x_{i}\right)\frac{\partial^{2}V(x)}{\partial x_{k}\partial x_{l}}\right|\leqslant c_{8},\quad\text{for all \ensuremath{x\in D,}}\label{eq:estimate diffusion}
\end{equation}
where $c_{8}>0$ is a constant.

Turning to the jump part $\mathcal{J}_{K}$, we define for $i\in I$
and $k\in\mathbb{N},$
\begin{align*}
\mathcal{J}_{k,i,\ast}V\left(x\right) & :=x_{i}\int_{\{0<\Vert\xi\Vert<k\}}\left(V\left(x+\xi\right)-V(x)-\langle\nabla V(x),\xi\rangle\right)\mu_{K,i}\left(\mathrm{d}\xi\right),
\end{align*}
and
\[
\mathcal{J}_{k,i}^{\ast}V\left(x\right):=x_{i}\int_{\{\Vert\xi\Vert\geqslant k\}}\left(V\left(x+\xi\right)-V(x)-\langle\nabla V(x),\xi\rangle\right)\mu_{K,i}\left(\mathrm{d}\xi\right).
\]
So $\mathcal{J}_{K}V(x)=\sum_{i\in I}(\mathcal{J}_{k,i,\ast}V(x)+\mathcal{J}{}_{k,i}^{\ast}V(x))$.

\textit{Big jumps.} By the mean value theorem, we get
\begin{align}
\left|\mathcal{J}_{k,i}^{\ast}V\left(x\right)\right| & \leqslant\left\Vert x_{i}\right\Vert \int_{\{\Vert\xi\Vert\geqslant k\}}\left(\left\Vert \nabla V\right\Vert _{\infty}\left\Vert \xi\right\Vert +\left\Vert \nabla V(x)\right\Vert \left\Vert \xi\right\Vert \right)\mu_{i}\left(\mathrm{d}\xi\right)\nonumber \\
 & \leqslant2\Vert x\Vert\left\Vert \nabla V\right\Vert _{\infty}\int_{\{\Vert\xi\Vert\geqslant k\}}\left\Vert \xi\right\Vert \mu_{i}\left(\mathrm{d}\xi\right)\label{eq: esti J*_k,V}\\
 & \leqslant c_{9}\left(1+V(x)\right)\int_{\{\Vert\xi\Vert\geqslant k\}}\left\Vert \xi\right\Vert \mu_{i}\left(\mathrm{d}\xi\right)<\infty,\nonumber
\end{align}
where we used that $\Vert\nabla V\Vert_{\infty}=\sup_{x\in D}\left\Vert \nabla V(x)\right\Vert <\infty$,
as a consequence of \eqref{eq:gradient of V}. Hence, by dominated
convergence, we can find large enough $k=k_{0}>0$ such that
\[
\left|\mathcal{J}_{k_{0},i}^{\ast}V\left(x\right)\right|\leqslant\frac{1}{2}c_{6}\left(1+V(x)\right),\quad x\in D.
\]

\textit{Small jumps.} To estimate the small jump part, we apply \eqref{eq: exact bound for nabla2}
and the mean value theorem, yielding for $\|x\|>3k_{0}$,
\begin{align}
\left|\mathcal{J}_{k_{0},i,\ast}V\left(x\right)\right| & \leqslant\left|x_{i}\int_{\{0<\Vert\xi\Vert<k_{0}\}}\left(\int_{0}^{1}\langle\nabla V\left(x+r\xi\right)-\nabla V(x),\xi\rangle\right)\mathrm{d}r\mu_{K,i}\left(\mathrm{d}\xi\right)\right|\nonumber \\
 & \leqslant\left\Vert x_{i}\right\Vert \sup_{\widetilde{x}\in B_{k_{0}}(x)}\left\Vert \nabla^{2}V\left(\widetilde{x}\right)\right\Vert \int_{\{0<\Vert\xi\Vert<k_{0}\}}\left\Vert \xi\right\Vert ^{2}\mu_{i}\left(\mathrm{d}\xi\right)\label{eq: esti J_k_0,iV}\\
 & \leqslant c_{7}\Vert x\Vert\sup_{\widetilde{x}\in B_{k_{0}}(x)}\frac{1}{V(\tilde{x})}\int_{\{0<\Vert\xi\Vert<k_{0}\}}\left\Vert \xi\right\Vert ^{2}\mu_{i}\left(\mathrm{d}\xi\right)\nonumber \\
 & \leqslant c_{10}\frac{\Vert x\Vert}{\Vert x\Vert-k_{0}}\leqslant2c_{10}<\infty,\nonumber
\end{align}
with some positive constant $c_{10}$ not depending on $K$. Here
$B_{k_{0}}(x)$ denotes the ball with center $x$ and radius $k_{0}$.
Note that $\mathcal{J}_{k_{0},i,\ast}V\left(x\right)$ is continuous
in $x\in D$. Hence, we conclude that
\[
\left|\mathcal{J}_{K}V(x)\right|\leqslant\frac{1}{2}c_{6}V(x)+c_{11},\quad x\in D.
\]
Combining the latter inequality with \eqref{eq:estimate drift} and
\eqref{eq:estimate diffusion}, we obtain the desired result, namely,
\[
\mathcal{A}_{K}^{\sharp}V(x)=\mathcal{D}V(x)+\mathcal{J}_{K}V(x)\leqslant-\frac{1}{2}c_{6}V(x)+c_{12},\quad x\in D.
\]
\end{proof}
\begin{rem}
For the function $V$ defined in the last lemma, we can easily find
positive constants $c_{1},c_{2},c_{3},c_{4}$ such that for all $x\in D$,
\begin{equation}
V(x)\leqslant c_{1}\|x\|+c_{2}\quad\mbox{and\ensuremath{\quad\|x\|\leqslant c_{3}V(x)+c_{4}.} }\label{eq: equiv of V and norm}
\end{equation}
\end{rem}

\begin{prop}
\label{prop:finiteness of sup expectation}Assume $m\geqslant1$ and
$n\geqslant1$. Suppose that $\beta\in\mathbb{M}_{d}^{-}$. Let $c,\thinspace C$
and $V$ be the same as in Lemma \ref{lem:foster-lyapunov estimate lemma}.
Then
\begin{equation}
\mathbb{E}_{x}\left[V\left(X_{K,t}\right)\right]\leqslant\mathrm{e}^{-ct}V(x)+c^{-1}C\quad\text{for all\  \ensuremath{K\geqslant1,\ }}x\in D\text{ and }t\in\mathbb{R}_{\geqslant0}.\label{eq:sup first moment existence}
\end{equation}
\end{prop}

\begin{proof}
Let $x\in D$, $K\geqslant1$ and $T>0$ be fixed. The proof is divided
into three steps.

\textit{Step 1:} We show that
\begin{equation}
\sup_{t\in[0,T]}\mathbb{E}_{x}\left[\|X_{K,t}\|^{2}\right]<\infty.\label{eq1: prop 3.5}
\end{equation}
 Since $\mu_{K,i}$ has compact support, it follows that $\int_{\left\{ \left\Vert \xi\right\Vert >1\right\} }\|\xi\|^{k}\mu_{K,i}(\mathrm{d}\xi)<\infty$
for all $k\in\mathbb{N}$. By \cite[Lemmas 5.3 and 6.5]{MR1994043},
we know that $\psi_{K}\in C^{2}(\mathbb{R}_{+}\times\mathcal{U})$.
Moreover, by \cite[Theorem 2.16]{MR1994043}, we have
\[
\mathbb{E}_{x}\left[\|X_{K,t}\|^{2}\right]=-\sum_{l=1}^{d}\left(\langle x,\partial_{\lambda_{l}}^{2}\psi_{K}(t,\mathrm{i}\lambda)|_{\lambda=0}\rangle+\langle x,\partial_{\lambda_{l}}\psi_{K}(t,\mathrm{i}\lambda)|_{\lambda=0}\rangle^{2}\right),
\]
where the right-hand side is a continuous function in $t\in[0,T]$.
So \eqref{eq1: prop 3.5} follows.

\textit{Step 2:} We show that
\begin{equation}
\sup_{t\in[0,T]}\mathbb{E}_{x}\left[V\left(X_{K,t}\right)\right]<\infty.\label{eq:finiteness of E(V(X_K,s))}
\end{equation}
In fact, \eqref{eq:finiteness of E(V(X_K,s))} follows from \eqref{eq: equiv of V and norm}
and \eqref{eq1: prop 3.5}.

\textit{Step 3:} We show that \eqref{eq:sup first moment existence}
is true. It follows from \cite[Theorem 2.12]{MR1994043} and \cite[Lemma 10.1]{MR1994043}
that
\begin{equation}
f\left(X_{K,t}\right)-f\left(X_{K,0}\right)-\int_{0}^{t}\mathcal{A}_{K}f\left(X_{K,s}\right)\mathrm{d}s,\quad t\in\mathbb{R}_{\geqslant0},\label{eq:martingale equation}
\end{equation}
is a $\mathbb{P}_{x}$-martingale for every $f\in C_{c}^{2}(D)$.
Note that $V$ belongs to $C^{2}(D)$ but does not have compact support.
Let $\varphi\in C_{c}^{\infty}(\mathbb{R}_{\geqslant0})$ be such
that $\mathbbm{1}_{[0,1]}\leqslant\varphi\leqslant\mathbbm{1}_{[0,2]}$,
and define $(\varphi_{j})_{j\geqslant1}\subset C_{c}^{\infty}(D)$
by $\varphi_{j}(y):=\varphi(\Vert y\Vert^{2}/j^{2})$. Then
\[
\varphi_{j}(y)=1\quad\text{for }\Vert y\Vert\leqslant j\quad\text{and}\quad\varphi_{j}(y)=0\quad\text{for }\Vert y\Vert>\sqrt{2}j,
\]
and $\varphi_{j}\to1$ as $j\to\infty$. For $j\in\mathbb{N}$, we
then define
\[
V_{j}(y):=V(y)\varphi_{j}(y),\quad y\in D.
\]
So $V_{j}\in C_{c}^{2}(D)$. In view of \eqref{eq:martingale equation}
and \cite[Chap.4, Lemma 3.2]{MR838085}, it follows that
\[
\mathrm{e}^{ct}V_{j}\left(X_{K,t}\right)-V_{j}\left(X_{K,0}\right)-\int_{0}^{t}\mathrm{e}^{cs}\mathcal{A}_{K}V_{j}\left(X_{K,s}\right)\mathrm{d}s-\int_{0}^{t}c\mathrm{e}^{cs}V_{j}\left(X_{K,s}\right)\mathrm{d}s,\quad t\in\mathbb{R}_{\geqslant0},
\]
is a $\mathbb{P}_{x}$-martingale, and hence
\begin{align*}
\mathrm{e}^{ct}\mathbb{E}_{x}\left[V_{j}\left(X_{K,t}\right)\right]-V_{j}\left(x\right) & =\mathbb{E}_{x}\left[\int_{0}^{t}\mathrm{e}^{cs}\left(\mathcal{A}_{K}V_{j}\left(X_{K,s}\right)+cV_{j}\left(X_{K,s}\right)\right)\mathrm{d}s\right].
\end{align*}
Now, a simple calculation shows
\[
\Vert\nabla\varphi_{j}(y)\Vert\leqslant\frac{2\Vert y\Vert}{j^{2}}\Vert\varphi'\Vert_{\infty}\leqslant\frac{2c_{1}\Vert y\Vert}{j^{2}},
\]
for some constant $c_{1}>0$. Therefore, by \eqref{eq: equiv of V and norm},
we get
\begin{align}
\Vert\nabla V_{j}(y)\Vert & =\mathbbm{1}_{\lbrace\Vert y\Vert\leqslant\sqrt{2}j\rbrace}\Vert\varphi_{j}(y)\nabla V(y)+V(y)\nabla\varphi_{j}(y)\Vert\nonumber \\
 & \leqslant\mathbbm{1}_{\lbrace\Vert y\Vert\leqslant\sqrt{2}j\rbrace}\left(\Vert\nabla V\Vert_{\infty}+c_{2}\left(1+\Vert y\Vert\right)\frac{2c_{1}\Vert y\Vert}{j^{2}}\right)\nonumber \\
 & \leqslant c_{3}\frac{\left(1+j\right)j}{j^{2}},\label{eq:gradient estimation for V_n}
\end{align}
where $c_{2}$ and $c_{3}$ are positive constants. A similar calculation
yields that there exists a constant $c_{4}>0$ such that
\begin{align*}
\left\Vert \nabla^{2}\varphi_{j}(y)\right\Vert  & \leqslant c_{4}\frac{\Vert y\Vert^{2}+j^{2}}{j^{4}}.
\end{align*}
So
\begin{align}
\Vert\nabla^{2}V_{j}(y)\Vert & \leqslant\mathbbm{1}_{\lbrace\Vert y\Vert\leqslant\sqrt{2}j\rbrace}\left(\Vert\nabla^{2}V\Vert_{\infty}+2\Vert\nabla V\Vert_{\infty}\Vert\nabla\varphi_{j}(y)\Vert+\Vert V(y)\Vert\Vert\nabla^{2}\varphi_{j}(y)\Vert\right)\nonumber \\
 & \leqslant\mathbbm{1}_{\lbrace\Vert y\Vert\leqslant\sqrt{2}j\rbrace}\left(c_{5}+\frac{c_{6}\Vert y\Vert}{j^{2}}+c_{7}(1+\Vert y\Vert)\frac{\Vert y\Vert^{2}+j^{2}}{j^{4}}\right)\nonumber \\
 & \leqslant c_{8}\frac{1+j+j^{2}}{j^{2}},\label{eq:gradient square estimation for V_n}
\end{align}
where $c_{5},\thinspace c_{6},\thinspace c_{7},\thinspace c_{8}>0$
are constants. Define $\mathcal{D}V_{j}$ and $\mathcal{J}_{K}V_{j}$
similarly as in \eqref{eq: defi DV} and \eqref{eq: defin J_KV},
respectively. It holds obviously that
\[
|\mathcal{D}V_{j}(y)|\leqslant c_{9}\|y\|\left(\Vert\nabla V_{j}\|_{\infty}+\Vert\nabla^{2}V_{j}\|_{\infty}\right),\quad y\in D.
\]
Similarly as in \eqref{eq: esti J*_k,V} and \eqref{eq: esti J_k_0,iV},
we have that for all $y\in D$,
\begin{align*}
|\mathcal{J}_{K}V_{j}(y)| & \leqslant c_{10}\|y\|\sum_{i=1}^{m}\Big(\Vert\nabla V_{j}\|_{\infty}\int_{\{\Vert\xi\Vert\geqslant1\}}\left\Vert \xi\right\Vert \mu_{i}\left(\mathrm{d}\xi\right)\\
 & \qquad\quad+\Vert\nabla^{2}V_{j}\|_{\infty}\int_{\{0<\Vert\xi\Vert<1\}}\left\Vert \xi\right\Vert ^{2}\mu_{i}\left(\mathrm{d}\xi\right)\Big).
\end{align*}
Using \eqref{eq:gradient estimation for V_n}, \eqref{eq:gradient square estimation for V_n}
and the above estimates for $\mathcal{D}V_{j}$ and $\mathcal{J}_{K}V_{j}$,
we obtain
\begin{equation}
\vert\mathcal{A}_{K}V_{j}(y)\vert\leqslant c_{11}(1+\|y\|),\quad y\in D,\label{eq: esti A_KV_J(y)}
\end{equation}
where $c_{11}>0$ is a constant not depending on $j$. The dominated
convergence theorem implies $\lim_{j\to\infty}\mathcal{A}_{K}V_{j}(y)=\mathcal{A}_{K}^{\sharp}V(y)$
for all $y\in D$. By \eqref{eq:finiteness of E(V(X_K,s))}, \eqref{eq: esti A_KV_J(y)}
and again dominated convergence, it follows that
\[
\mathrm{e}^{ct}\mathbb{E}_{x}\left[V\left(X_{K,t}\right)\right]-V\left(x\right)=\mathbb{E}_{x}\left[\int_{0}^{t}\mathrm{e}^{cs}\left(\mathcal{A}_{K}^{\sharp}V\left(X_{K,s}\right)+cV\left(X_{K,s}\right)\right)\mathrm{d}s\right].
\]
Applying Lemma \ref{lem:foster-lyapunov estimate lemma} yields
\[
\mathrm{e}^{ct}\mathbb{E}_{x}\left[V\left(X_{K,t}\right)\right]-V\left(x\right)\leqslant\mathbb{E}_{x}\left[\int_{0}^{t}\mathrm{e}^{cs}C\mathrm{d}s\right]\leqslant c^{-1}C\mathrm{e}^{ct},
\]
which implies
\[
\mathbb{E}_{x}\left[V\left(X_{K,t}\right)\right]\leqslant\mathrm{e}^{-ct}V(x)+c^{-1},\quad\mbox{for}t\in[0,T].
\]
Since $x\in D$, $K\geqslant1$ and $T>0$ are arbitrary, the assertion
follows.
\end{proof}
Arguing similarly as in Lemma \ref{lem:foster-lyapunov estimate lemma}
and Proposition \ref{prop:finiteness of sup expectation}, we obtain
also an analog result for the case where $m\geqslant1$ and $n=0$.
\begin{prop}
\label{lem: case m>0 and n=00003D0}Assume $m\geqslant1$ and $n=0$.
Suppose that $\beta\in\mathbb{M}_{d}^{-}$. Let $V\in C^{2}(D,\mathbb{R})$
be such that $V>0$ on $D$ and
\[
V(x)=\langle x,x\rangle_{I}^{1/2},\quad\text{whenever }\Vert x\Vert>2.
\]
Then $\mathcal{A}_{K}^{\sharp}V$ is well-defined and there exist
positive constants $c$ and $C$, independent of $K$, such that
\[
\mathcal{A}_{K}^{\sharp}V(x)\leqslant-cV(x)+C,\quad\forall\ensuremath{x\in D.}
\]
 Moreover, for all $K\geqslant1$, $t\geqslant0$ and $x\in D$, it
holds that
\[
\mathbb{E}_{x}\left[V\left(X_{K,t}\right)\right]\leqslant\mathrm{e}^{-ct}V(x)+c^{-1}C.
\]
\end{prop}

We are now ready to prove the uniform boundedness for the first moment
of $X_{t}$, $t\geqslant0$.
\begin{prop}
\label{thm:uniform boundedness for the first moment of X_t}Let $X$
be an affine process satisfying \emph{(\ref{eq: assum sect. 3})}.
Suppose that $\beta\in\mathbb{M}_{d}^{-}$. Then
\begin{equation}
\sup_{t\geqslant0}\mathbb{E}_{x}\left[\left\Vert X_{t}\right\Vert \right]<\infty\quad\mbox{for all \ensuremath{x\in D.}}\label{eq: to prove, prop. 3.7}
\end{equation}
\end{prop}

\begin{proof}
If $m=0$ and $n\geqslant1$, then $(X_{t})_{t\geqslant0}$ degenerates
to a deterministic motion governed by the vector field $x\mapsto\beta x$.
In this case we have
\[
X_{t}=\mathrm{e}^{\beta t}X_{0},
\]
so \eqref{eq: to prove, prop. 3.7} follows from the assumption that
$\beta\in\mathbb{M}_{d}^{-}$.

For the case where $m\geqslant1$, by Propositions \ref{prop:finiteness of sup expectation}
and \ref{lem: case m>0 and n=00003D0}, we have
\begin{equation}
\mathbb{E}_{x}\left[V\left(X_{K,t}\right)\right]\leqslant\mathrm{e}^{-ct}V(x)+c^{-1}C,\quad\text{for all }K\geqslant1,x\in D\ \mbox{and }t\in\mathbb{R}_{\geqslant0},\label{eq 2: prop. 3.7}
\end{equation}
where $c,\thinspace C>0$ are constants not depending on $K$.

Let $x\in D$ be fixed and assume without loss of generality that
$X_{0}=x$ a.s. In view of Lemma \ref{lem:approximation lemma} and
Skorokhod's representation theorem (see, e.g., \cite[Chap.3, Theorem 1.8]{MR838085}),
there exist some probability space $(\widetilde{\Omega},\widetilde{\mathcal{F}},\widetilde{\mathbb{P}})$
on which $(\widetilde{X}{}_{K,t})_{K\geqslant1}$ and $\widetilde{X}{}_{t}$
are defined such that $\widetilde{X}{}_{K,t}$ and $\widetilde{X}{}_{t}$
have the same distributions as $X_{K,t}$ and $X_{t}$, respectively,
and $\widetilde{X}{}_{K,t}\to\widetilde{X}{}_{t}$ $\tilde{\mathbb{P}}$-almost
surely as $K\to\infty$. Hence $V(\widetilde{X}{}_{K,t})\to V(\widetilde{X}{}_{t})$
$\mathbb{\tilde{P}}$-almost surely as $K\to\infty$.\textcolor{red}{{}
}By \eqref{eq 2: prop. 3.7} and Fatou's lemma, we have
\begin{align*}
\mathbb{E}_{x}\left[V\left(X_{t}\right)\right]=\widetilde{\mathbb{E}}\left[V\left(\widetilde{X}{}_{t}\right)\right] & \leqslant\liminf_{K\to\infty}\widetilde{\mathbb{E}}\left[V\left(\widetilde{X}{}_{K,t}\right)\right]\\
 & =\liminf_{K\to\infty}\mathbb{E}_{x}\left[V\left(X_{K,t}\right)\right]\\
 & \leqslant\mathrm{e}^{-ct}V\left(x\right)+c^{-1}C
\end{align*}
for all $t\geqslant0$. By \eqref{eq: equiv of V and norm}, the assertion
follows.
\end{proof}

\subsection{Exponential convergence of \textcolor{black}{$\psi(t,u)$ to zero}}

\textcolor{black}{In this subsection we study the convergence speed
of }$\psi(t,u)\to0$ as $t\to\infty$.
\begin{lem}
\label{prop:convergence of psi in a neighborhood of zero}Suppose
that $\beta\in\mathbb{M}_{d}^{-}$. There exist $\delta>0$ and constants
$C_{1},C_{2}>0$ such that for all $u\in\mathcal{U}$ with $\|u\|<\delta$,
\begin{equation}
\left\Vert \psi\left(t,u\right)\right\Vert \leqslant C_{1}\exp\left\{ -C_{2}t\right\} ,\quad t\geqslant0.\label{eq0: lemma 3.8}
\end{equation}
\end{lem}

\begin{proof}
For $u\in\mathcal{U}$, we can write $u=(v,w)\in\mathbb{C}_{\leqslant0}^{m}\times\mathrm{i}\mathbb{R}^{n}$
and further $v=x+\mathrm{i}y$ and $w=\mathrm{i}z$, where $x\in\mathbb{R}_{\leqslant0}^{m}$,
$y\in\mathbb{R}^{m}$ and $z\in\mathbb{R}^{n}$. Therefore,
\[
\psi(t,u)=\psi\left(t,v,w\right)=\begin{pmatrix}\psi^{I}\left(t,x+\mathrm{i}y,\mathrm{i}z\right)\\
\mathrm{i}\mathrm{e}^{\beta_{JJ}^{\top}t}z
\end{pmatrix}.
\]
For $x\in\mathbb{R}_{\leqslant0}^{m}$, $y\in\mathbb{R}^{m}$, and
$z\in\mathbb{R}^{n}$, we define
\[
\widetilde{\psi}\left(t,x,y,z\right):=\begin{pmatrix}\mathrm{Re}\thinspace\psi^{I}\left(t,x+\mathrm{i}y,\mathrm{i}z\right)\\
\mathrm{Im}\thinspace\psi^{I}\left(t,x+\mathrm{i}y,\mathrm{i}z\right)\\
\mathrm{e}^{\beta_{JJ}^{\top}t}z
\end{pmatrix}=\begin{pmatrix}\vartheta\\
\eta\\
\zeta
\end{pmatrix},\quad t\geqslant0.
\]
Recall that $\psi^{I}(t,u)$ satisfies the Riccati equation
\[
\partial_{t}\psi^{I}(t,v,w)=R^{I}\left(\psi^{I}(t,v,w),\mathrm{e}^{\beta_{JJ}^{\top}t}w\right),\quad\text{ }\psi^{I}(0,v,w)=v.
\]
So
\begin{align*}
\partial_{t}\widetilde{\psi}(t,x,y,z) & =\begin{pmatrix}\partial_{t}\mathrm{Re}\thinspace\psi^{I}\left(t,x+\mathrm{i}y,\mathrm{i}z\right)\\
\partial_{t}\mathrm{Im}\thinspace\psi^{I}\left(t,x+\mathrm{i}y,\mathrm{i}z\right)\\
\partial_{t}\mathrm{e}^{\beta_{JJ}^{\top}t}z
\end{pmatrix}\\
 & =\begin{pmatrix}\mathrm{Re}\thinspace R^{I}\left(\psi^{I}\left(t,x+\mathrm{i}y,\mathrm{i}z\right),\mathrm{i}\mathrm{e}^{\beta_{JJ}^{\top}t}z\right)\\
\mathrm{Im}\thinspace R^{I}\left(\psi^{I}\left(t,x+\mathrm{i}y,\mathrm{i}z\right),\mathrm{i}\mathrm{e}^{\beta_{JJ}^{\top}t}z\right)\\
\beta_{JJ}^{\top}\mathrm{e}^{\beta_{JJ}^{\top}t}z
\end{pmatrix}\\
 & =\begin{pmatrix}\mathrm{Re}\thinspace R^{I}\left(\mathrm{Re}\thinspace\psi^{I}\left(t,x+\mathrm{i}y,\mathrm{i}z\right)+\mathrm{i}\mathrm{Im}\thinspace\psi^{I}\left(t,x+\mathrm{i}y,\mathrm{i}z\right),\mathrm{i}\mathrm{e}^{\beta_{JJ}^{\top}t}z\right)\\
\mathrm{Im}\thinspace R^{I}\left(\mathrm{Re}\thinspace\psi^{I}\left(t,x+\mathrm{i}y,\mathrm{i}z\right)+\mathrm{i}\mathrm{Im}\thinspace\psi^{I}\left(t,x+\mathrm{i}y,\mathrm{i}z\right),\mathrm{i}\mathrm{e}^{\beta_{JJ}^{\top}t}z\right)\\
\beta_{JJ}^{\top}\mathrm{e}^{\beta_{JJ}^{\top}t}z
\end{pmatrix}\\
 & =\begin{pmatrix}\mathrm{Re}\thinspace R^{I}\left(\vartheta+\mathrm{i}\eta,\mathrm{i}\zeta\right)\\
\mathrm{Im}\thinspace R^{I}\left(\vartheta+\mathrm{i}\eta,\mathrm{i}\zeta\right)\\
\beta_{JJ}^{\top}\zeta
\end{pmatrix}\\
 & =:\widetilde{R}\left(\vartheta,\eta,\zeta\right),
\end{align*}
where the map $\mathbb{R}_{\leqslant0}^{m}\times\mathbb{R}^{m}\times\mathbb{R}^{n}\ni(\vartheta,\eta,\zeta)\mapsto\widetilde{R}\left(\vartheta,\eta,\zeta\right)$
is $C^{1}$ by \cite[Lemma 5.3]{MR1994043}. Hence $\widetilde{\psi}(t,x,y,z)$
solves the equation
\begin{equation}
\partial_{t}\widetilde{\psi}(t,x,y,z)=\widetilde{R}\left(\widetilde{\psi}(t,x,y,z)\right),\quad t\geqslant0,\quad\psi(0,x,y,z)=(x,y,z).\label{eq:riccati equation for transformed psi}
\end{equation}
Similarly to \cite[p.1011, (6.7)]{MR1994043}, we have, for $u=(x+\mathrm{i}y,\mathrm{i}z),$
\begin{align}
\mathrm{Re}\thinspace R_{i}\left(x+\mathrm{i}y,\mathrm{i}z\right) & =\alpha_{i,ii}x_{i}^{2}-\langle\alpha_{i}\mathrm{Im}\thinspace u,\mathrm{Im}\thinspace u\rangle+\sum_{k=1}^{m}\beta_{ki}x_{k}\nonumber \\
 & \quad+\int_{D\backslash\lbrace0\rbrace}\left(\mathrm{e}^{\langle\xi_{I},x\rangle}\cos\langle\mathrm{Im}\thinspace u,\xi\rangle-1-\langle\xi_{I},x\rangle\right)\mu_{i}\left(\mathrm{d}\xi\right)\label{eq: form of real part of R_i}
\end{align}
and
\begin{align}
\mathrm{Im}\ R_{i}\left(x+\mathrm{i}y,\mathrm{i}z\right) & =2\alpha_{i,ii}x_{i}y_{i}+\langle\beta_{Ii},y\rangle+\langle\beta_{Ji},z\rangle\nonumber \\
 & \quad+\int_{D\backslash\lbrace0\rbrace}\left(\mathrm{e}^{\langle\xi_{I},x\rangle}\sin\langle\mathrm{Im}\thinspace u,\xi\rangle-\langle\mathrm{Im}\thinspace u,\xi\rangle\right)\mu_{i}\left(\mathrm{d}\xi\right).\label{eq: form of imag part of R_i}
\end{align}
Since $\widetilde{R}:\mathbb{R}_{\leqslant0}^{m}\times\mathbb{R}^{m+n}\to\mathbb{R}^{2m+n}$
is $C^{1}$, so
\begin{align}
 & \left\Vert \widetilde{R}\left(\vartheta,\eta,\zeta\right)-D\widetilde{R}(\mathbf{0})\left(\vartheta,\eta,\zeta\right)^{\top}\right\Vert \nonumber \\
 & \quad\quad\quad\quad=\left\Vert \widetilde{R}\left(\vartheta,\eta,\zeta\right)-\widetilde{R}(\mathbf{0})-D\widetilde{R}(\mathbf{0})\left(\vartheta,\eta,\zeta\right)^{\top}\right\Vert \nonumber \\
 & \quad\quad\quad\quad=\left\Vert \int_{0}^{1}D\widetilde{R}\left(r\left(\vartheta,\eta,\zeta\right)\right)\left(\vartheta,\eta,\zeta\right)^{\top}\mathrm{d}r-\int_{0}^{1}D\widetilde{R}(\mathbf{0})\left(\vartheta,\eta,\zeta\right)^{\top}\mathrm{d}r\right\Vert \nonumber \\
 & \quad\quad\quad\quad\leqslant\sup_{0\leqslant r\leqslant1}\left\Vert D\widetilde{R}\left(r\left(\vartheta,\eta,\zeta\right)\right)-D\widetilde{R}(\mathbf{0})\right\Vert \cdot\left\Vert \left(\vartheta,\eta,\zeta\right){}^{\top}\right\Vert \nonumber \\
 & \quad\quad\quad\quad=o\left(\left\Vert \left(\vartheta,\eta,\zeta\right)^{\top}\right\Vert \right)\label{eq:g convergence}
\end{align}
holds. Here, $D\widetilde{R}(\vartheta,\eta,\zeta)$ denotes the Jacobian,
i.e., the matrix consisting of all first-order partial derivatives
of the vector-valued function $(\vartheta,\eta,\zeta)\mapsto\widetilde{R}(\vartheta,\eta,\zeta)$.
According to \eqref{eq: form of real part of R_i} and \eqref{eq: form of imag part of R_i},
we see that\textcolor{red}{{} ${\normalcolor D\widetilde{R}(\mathbf{0})}$
}is a matrix taking the form

\renewcommand*\arraystretch{1.4}\[
D\widetilde{R}(\mathbf{0})=
\left( \begin{array}{cc;{2pt/2pt}c}
\beta^{\top}_{II} & 0 & 0 \\
0 & \beta^{\top}_{II} & \ast \\ \hdashline[2pt/2pt]
0 & 0 & \beta^{\top}_{JJ}
\end{array} \right)
\]where $\ast$ is a ($m\times n$)-matrix. By the Riccati equation
\eqref{eq:riccati equation for transformed psi} for $\widetilde{\psi}$,
we can write
\[
\partial_{t}\widetilde{\psi}\left(t,x,y,z\right)=D\widetilde{R}(\mathbf{0})\widetilde{\psi}\left(t,x,y,z\right)+\left(\widetilde{R}\left(\widetilde{\psi}(t,x,y,z)\right)-D\widetilde{R}(\mathbf{0})\widetilde{\psi}\left(t,x,y,z\right)\right).
\]
From \eqref{eq:g convergence} it follows that
\[
\lim_{\left\Vert (\vartheta,\eta,\zeta)\right\Vert \to0}\frac{\left\Vert \widetilde{R}\left(\vartheta,\eta,\zeta\right)-D\widetilde{R}(\mathbf{0})\left(\vartheta,\eta,\zeta\right)^{\top}\right\Vert }{\left\Vert (\vartheta,\eta,\zeta)\right\Vert }=0.
\]
By assumption, we know that $\beta_{II}\in\mathbb{M}_{m}^{-}$ and
$\beta_{JJ}\in\mathbb{M}_{n}^{-}$, which ensures $D\widetilde{R}(\mathbf{0})\in\mathbb{M}_{2m+n}^{-}$.
Now, an application of the linearized stability theorem (see, e.g.,
\cite[VII. Stability Theorem, p.311]{MR1629775}) yields that $\widetilde{\psi}$
is asymptotically stable at $\mathbf{0}$. Moreover, as shown in the
proof of \cite[VII. Stability Theorem, p.311]{MR1629775}, we can
find constants $\delta,\thinspace c_{1},\thinspace c_{2}>0$ such
that
\[
\left\Vert \widetilde{\psi}(t,x,y,z)\right\Vert \leqslant c_{1}\mathrm{e}^{-c_{2}t},\quad\forall\text{ }t\geqslant0,\thinspace(x,y,z)\in B_{\delta}(0)\cap\mathbb{R}_{\leqslant0}^{m}\times\mathbb{R}^{m+n},
\]
where $B_{\delta}(0)$ denotes the ball with center $0$ and radius
$\delta$. By the definition of $\widetilde{\psi}$, the latter inequality
implies that \eqref{eq0: lemma 3.8} is true. The lemma is proved.
\end{proof}

Next, we extend the estimate in Lemma \ref{prop:convergence of psi in a neighborhood of zero}
to all $u\in\mathcal{U}$.
\begin{prop}
\label{thm:exponential convergence of psi for all u}Let $X$ be an
affine process satisfying \emph{\eqref{eq: assum sect. 3}}. Suppose
that $\beta\in\mathbb{M}_{d}^{-}$. Then for every $u\in\mathcal{U},$
there exist positive constants $c_{1},c_{2}$, which depend on $u$,
such that
\[
\left\Vert \psi\left(t,u\right)\right\Vert \leqslant c_{1}\exp\left\{ -c_{2}t\right\} ,\quad t\geqslant0.
\]
\end{prop}

\begin{proof}
Our proof is inspired by the proof of \cite[Theorem 2.4]{MR2599675}.
By Proposition \ref{thm:uniform boundedness for the first moment of X_t},
we have $\sup_{t\in\mathbb{R}_{\geqslant0}}\mathbb{E}_{x}[\Vert X_{t}\Vert]<\infty$
for all $x\in D$. Then for $M>0$,
\[
\mathbb{\mathbb{P}}_{x}\left(\left\Vert X_{t}\right\Vert >M\right)\leqslant\frac{\mathbb{E}_{x}\left[\left\Vert X_{t}\right\Vert \right]}{M}\leqslant\frac{\sup_{t\geqslant0}\mathbb{E}_{x}\left[\left\Vert X_{t}\right\Vert \right]}{M},
\]
which implies
\[
\sup_{t\geqslant0}\mathbb{P}_{x}\left(\left\Vert X_{t}\right\Vert >M\right)\to0\quad\mbox{as \ensuremath{M\to\infty}}.
\]
We see that under $\mathbb{P}_{x}$, the sequence $\{X_{t},\,t\geqslant0\}$
is tight. Consider an arbitrary subsequence $\{X_{t'}\}$. Then it
contains a further subsequence $\{X_{t''}\}$ converging in law to
some limiting random vector, say $X^{a}$. Since $X_{t''}$ converges
weakly to $X^{a}$ as $t''\to\infty$, Lévy's continuity theorem implies
that the characteristic function of $X_{t''}$ converges pointwise
to that of $X^{a}$, namely,
\[
\lim_{t''\rightarrow\infty}\mathrm{\mathbb{E}}_{x}\left[\exp\left\{ \langle u,X_{t''}\rangle\right\} \right]=\mathbb{E}\left[\exp\left\{ \langle u,X^{a}\rangle\right\} \right],\quad\text{for all }u\in\mathcal{U}.
\]
\textcolor{black}{We know by Proposition} \ref{prop:convergence of psi in a neighborhood of zero}\textcolor{black}{{}
that the original sequence} $\{X_{t}\}$\textcolor{black}{{} satisfies
}
\[
\lim_{t\to\infty}\mathbb{E}_{x}\left[\exp\left\{ \langle u,X_{t}\rangle\right\} \right]=\lim_{t\to\infty}\exp\left\{ \langle x,\psi(t,u)\rangle\right\} =1
\]
\textcolor{black}{for }all $u\in\mathcal{U}$ with $\|u\|<\delta$.
As a consequence, we get
\begin{equation}
\mathbb{E}\left[\exp\left\{ \langle u,X^{a}\rangle\right\} \right]=1,\quad\text{for all }u\in\mathcal{U}\quad\mbox{with\quad}\|u\|<\delta.\label{eq: look for contra}
\end{equation}
We claim that $X^{a}=0$ almost surely. To prove this, we consider
an arbitrary $z\in\Rd$ with $z\neq0$. Then there exists an $u_{0}\in\Rd$
with $\|u_{0}\|<\delta$ such that $0<\langle u_{0},z\rangle<\pi/6$,
and hence $0<\cos(\langle u_{0},z\rangle)<1$. Continuity of cosinus
implies that there exists an $\varepsilon>0$ such that $0\not\in B_{\varepsilon}(z):=\lbrace y\in\mathbb{R}^{d}\thinspace:\thinspace\Vert y-z\Vert<\varepsilon\rbrace$
and $0<\cos(\langle u_{0},y\rangle)<1$ for all $y\in B_{\varepsilon}(z)$.
Suppose that $\mathbb{P}\left(X^{a}\in B_{\varepsilon}(z)\right)>0$.
It follows that
\[
\mathbb{E}\left[\cos\left(\langle u_{0},X^{a}\rangle\right)\mathbbm{1}_{\lbrace X^{a}\in B_{\varepsilon}(z)\rbrace}\right]<\mathbb{P}\left(X^{a}\in B_{\varepsilon}(z)\right),
\]
which in turn implies
\begin{align*}
\mathrm{Re}\thinspace\mathbb{E}\left[\exp\left\{ \mathrm{i}\langle u_{0},X^{a}\rangle\right\} \right] & =\mathbb{E}\left[\cos\left(\langle u_{0},X^{a}\rangle\right)\right]\\
 & \leqslant\mathbb{E}\left[\cos\left(\langle u_{0},X^{a}\rangle\right)\mathbbm{1}_{\lbrace X^{a}\in B_{\varepsilon}(z)\rbrace}\right]\\
 & \quad+\mathbb{E}\left[\cos\left(\langle u_{0},X^{a}\rangle\right)\mathbbm{1}_{\lbrace X^{a}\not\in B_{\varepsilon}(z)\rbrace}\right]\\
 & <\mathbb{P}\left(X^{a}\in B_{\varepsilon}(z)\right)+\mathbb{P}\left(X^{a}\not\in B_{\varepsilon}(z)\right)\\
 & =1,
\end{align*}
a contradiction to \eqref{eq: look for contra}. We conclude that
$\mathbb{P}(X^{a}\in B_{\varepsilon}(z))=0$. Since $z\neq0$ is arbitrary,
$X^{a}$ must be $0$ almost surely. Now we have shown that every
subsequence of $\{X_{t}\}$ contains a further subsequence converging
weakly to $\delta_{0}$, so the original sequence $\{X_{t}\}$ must
converge to $\delta_{0}$ weakly. In view of this, we now denote $X^{a}$
by $X_{\infty}$ which is $0$ almost surely. We have thus shown that
for all $x\in D$ and $u\in\mathcal{U},$
\begin{equation}
\exp\left\{ \langle x,\psi(t,u)\rangle\right\} =\mathbb{E}_{x}\left[\exp\left\{ \langle u,X_{t}\rangle\right\} \right]\to1\quad\text{as\ensuremath{\quad}}t\to\infty.\label{eq:pointwise convergence of characterisitc function}
\end{equation}

From the above convergence of $\exp\left\{ \langle x,\psi(t,u)\rangle\right\} $
to $1$, we infer that for each $i=1,\ldots,d,$
\begin{equation}
\mathrm{Re}\thinspace\psi_{i}(t,u)\to0\quad\text{as }t\to\infty.\label{eq:convergence of re psi_i}
\end{equation}
Moreover, we must have $\sup_{t\in[0,\infty)}|\psi_{i}(t,u)|\leqslant C$
for some constant $C=C(u)<\infty$, otherwise, by continuity, $\mathrm{Im}\thinspace\psi_{i}(t,u)$
hits the set $\left\{ 2k\pi+\pi/2:\ k\in\mathbb{Z}\right\} $ infinitely
many times as $t\to\infty$, so $\sin\left(\mathrm{Im}\thinspace\psi_{i}(t,u)\right)=1$
infinitely often, contradicting the fact that $\exp\left\{ \langle x,\psi(t,u)\rangle\right\} \to1$
for all $x\in D$.

Let $z,\thinspace z'\in\mathbb{C}$ be two different accumulation
points of $\{\psi_{1}(t,u),\ t\geqslant0\}$ as $t\to\infty$, that
is, we can find sequences $t_{n},t'_{n}\to\infty$ such that $\psi_{1}(t_{n},u)\to z$
and $\psi_{1}(t'_{n},u)\to z'$. Using once again the convergence
in \eqref{eq:pointwise convergence of characterisitc function}, we
obtain that $z=\mathrm{i}2\pi k_{1}$ and $z'=\mathrm{i}2\pi k_{2}$
for some $k_{1},k_{2}\in\mathbb{Z}$. By \eqref{eq:convergence of re psi_i}
and a similar argument as in the last paragraph, $\psi_{1}(t,u)$
is not allowed to fluctuate between $z$ and $z'$, showing that $z=z'$.
So $z=\mathrm{i}2\pi k_{1}$ is the only accumulation point of $\{\psi_{1}(t,u),\ t\geqslant0\}$,
and $\psi_{1}(t,u)\to z=\mathrm{i}2\pi k_{1}$ as $t\to\infty$. Moreover,
we must have $k_{1}=0$, otherwise for some $x\in D$ we get $\exp\lbrace x_{1}2\pi\mathrm{i}k_{1}\rbrace\not=1$,
which is impossible due to \eqref{eq:pointwise convergence of characterisitc function}.
We conclude that
\[
\psi_{1}\left(t,u\right)\to0\quad\text{as }t\to\infty\text{ for all }u\in\mathcal{U}.
\]
In the same way it follows that $\psi_{i}\left(t,u\right)\to0$ as
$t\to0$ for all $i=2,\ldots,d$ and $u\in\mathcal{U}$.

Finally, we prove that the convergence of $\psi(t,u)$ to zero as
$t\to\infty$ is exponentially fast. Since $\psi(t,u)$ converges
to $0$ as $t\to\infty$, there exists a $t_{0}>0$ such that $\left\Vert \psi(t_{0},u)\right\Vert <\delta$.
Combining Lemma \ref{prop:convergence of psi in a neighborhood of zero}
with the semi-flow property of $\psi$, we conclude that
\[
\left\Vert \psi\left(t+t_{0},u\right)\right\Vert =\left\Vert \psi\left(t,\psi\left(t_{0},u\right)\right)\right\Vert \leqslant c_{1}\mathrm{e}^{-c_{2}t},\quad t\geqslant0,
\]
for some positive constants $c_{1}$ and $c_{2}$. Hence,
\[
\left\Vert \psi\left(t,u\right)\right\Vert \leqslant c_{3}\mathrm{e}^{-c_{2}t},\quad t\geqslant t_{0}.
\]
Since $\sup_{t\in[0,t_{0}]}\Vert\psi(t,u)\Vert<c_{4}$, where $c_{4}>0$
is a constant, it follows that
\[
\left\Vert \psi\left(t,u\right)\right\Vert \leqslant c_{5}\mathrm{e}^{-c_{2}t},\quad t\geqslant0,
\]
with another constant $c_{5}>0$. This completes our proof.
\end{proof}

\section{Proof of the main result}

In this section we will prove Theorem \ref{thm:stationarity of affine processes}.\\

Let $X$ be an affine process with state space $D$ and admissible
parameters $(a,\alpha,b,\beta,m,\mu)$. Recall that $F(u)$ is given
by \eqref{eq:representation of F(u)}. We start with the following
lemma.
\begin{lem}
\label{lem:finiteness of limit F(psi(s,u))}Suppose $\beta\in\mathbb{M}_{d}^{-}$
and $\int_{\left\{ \left\Vert \xi\right\Vert >1\right\} }\log\left\Vert \xi\right\Vert m\left(\mathrm{d}\xi\right)<\infty$.
Then
\[
\int_{0}^{\infty}\left|F\left(\psi\left(s,u\right)\right)\right|\mathrm{d}s<\infty\quad\text{for all }u\in\mathcal{U}.
\]
\end{lem}

\begin{proof}
Let $u\in\mathcal{U}$ be fixed. By Remark \ref{rem:The-assumption-that}
and Proposition \ref{thm:exponential convergence of psi for all u},
we can find constants $c_{1},\thinspace c_{2}>0$ depending on $u$
such that
\begin{equation}
\left\Vert \psi(s,u)\right\Vert \leqslant c_{1}\mathrm{e}^{-c_{2}s},\quad s\geqslant0.\label{eq:psi exp ineq}
\end{equation}
It is clear that finiteness of $\int_{0}^{\infty}\left|F\left(\psi\left(s,u\right)\right)\right|\mathrm{d}s$
depends only on the jump part of $F$. We define
\begin{align*}
\mathcal{I}\left(u\right) & =\int_{0}^{\infty}\int_{\left\{ 0<\left\Vert \xi\right\Vert \leqslant1\right\} }\left|\mathrm{e}^{\langle\xi,\psi(s,u)\rangle}-1-\langle\psi^{J}(s,u),\xi_{J}\rangle\right|m\left(\mathrm{d}\xi\right)\mathrm{d}s\\
 & \quad+\int_{0}^{\infty}\int_{\left\{ \left\Vert \xi\right\Vert >1\right\} }\left|\mathrm{e}^{\langle\xi,\psi(s,u)\rangle}-1\right|m\left(\mathrm{d}\xi\right)\mathrm{d}s\\
 & =:\mathcal{I}_{\ast}\left(u\right)+\mathcal{I}^{\ast}\left(u\right).
\end{align*}
With the latter fact in mind, we start with the big jumps. We can
apply Fubini's theorem to get
\begin{align*}
\mathcal{I}^{\ast}\left(u\right) & =\int_{\left\{ \left\Vert \xi\right\Vert >1\right\} }\int_{0}^{\infty}\left|\mathrm{e}^{\langle\xi,\psi(s,u)\rangle}-1\right|\mathrm{d}sm\left(\mathrm{d}\xi\right).
\end{align*}
Let us define $I_{1}\left(\xi\right):=\int_{0}^{\infty}\left|\exp\lbrace\langle\psi(s,u),\xi\rangle\rbrace-1\right|\mathrm{d}s$.
For $\|\xi\|>1$, by a change of variables $t:=\exp\left\{ -c_{2}s\right\} \left\Vert \xi\right\Vert $,
we get $\mathrm{d}s=-c_{2}^{-1}t^{-1}\mathrm{d}t$, and hence
\begin{align*}
I_{1}\left(\xi\right) & =-\frac{1}{c_{2}}\int_{\left\Vert \xi\right\Vert }^{0}\frac{1}{t}\left|\mathrm{e}^{\langle\xi,\psi\left(s^{-1}(t),u\right)\rangle}-1\right|\mathrm{d}t\\
 & =\frac{1}{c_{2}}\int_{0}^{\left\Vert \xi\right\Vert }\frac{1}{t}\left|\mathrm{e}^{\langle\xi,\psi\left(s^{-1}(t),u\right)\rangle}-1\right|\mathrm{d}t\\
 & \leqslant\frac{1}{c_{2}}\int_{0}^{1}\frac{1}{t}\left|\mathrm{e}^{\langle\xi,\psi\left(s^{-1}(t),u\right)\rangle}-1\right|\mathrm{d}t+\frac{1}{c_{2}}\int_{1}^{\left\Vert \xi\right\Vert }\frac{2}{t}\mathrm{d}t\\
 & =:I_{2}\left(\xi\right)+I_{3}\left(\xi\right).
\end{align*}
Note that
\begin{align*}
\left|\mathrm{e}^{\langle\xi,\psi\left(s^{-1}(t),u\right)\rangle}-1\right| & =\left|\int_{0}^{1}\mathrm{e}^{r\langle\xi,\psi\left(s^{-1}(t),u\right)\rangle}\langle\xi,\psi\left(s^{-1}(t),u\right)\rangle\mathrm{d}r\right|\\
 & \leqslant\left|\langle\xi,\psi\left(s^{-1}(t),u\right)\rangle\right|.
\end{align*}
Using \eqref{eq:psi exp ineq}, we obtain
\begin{align*}
I_{2}\left(\xi\right) & \leqslant\frac{1}{c_{2}}\int_{0}^{1}\frac{1}{t}\left|\langle\psi\left(s^{-1}(t),u\right),\xi\rangle\right|\mathrm{d}t\\
 & \leqslant\frac{1}{c_{2}}\int_{0}^{1}\frac{1}{t}\left\Vert \psi\left(s^{-1}(t),u\right)\right\Vert \left\Vert \xi\right\Vert \mathrm{d}t\\
 & \leqslant\frac{1}{c_{2}}\int_{0}^{1}\frac{c_{1}}{t}\mathrm{e}^{-c_{2}s^{-1}(t)}\left\Vert \xi\right\Vert \mathrm{d}t.
\end{align*}
Since $s^{-1}(t)=\log(t\Vert\xi\Vert^{-1})(-c_{2})^{-1}$, it follows
that
\[
I_{2}\left(\xi\right)\leqslant\frac{1}{c_{2}}\int_{0}^{1}c_{1}\mathrm{d}t=\frac{c_{1}}{c_{2}}.
\]
On the other hand, it is easy to see that
\[
I_{3}\left(\xi\right)\leqslant\frac{2}{c_{2}}\log\left\Vert \xi\right\Vert ,
\]
Having established the latter inequalities, we conclude that
\begin{align*}
{\normalcolor \left|\mathcal{I}^{\ast}\left(u\right)\right|} & {\normalcolor \leqslant\int_{\left\{ \left\Vert \xi\right\Vert >1\right\} }\left(I_{2}\left(\xi\right)+I_{3}\left(\xi\right)\right)m\left(\mathrm{d}\xi\right)}\\
{\normalcolor } & {\color{black}{\normalcolor {\color{red}{\normalcolor {\normalcolor }\leqslant\int_{\left\{ \left\Vert \xi\right\Vert >1\right\} }\left(\frac{c_{1}}{c_{2}}+\frac{2}{c_{2}}\log\left\Vert \xi\right\Vert \right)m\left(\mathrm{d}\xi\right)}}}}\\
{\normalcolor } & {\color{black}=\frac{c_{1}}{c_{2}}m\left(\left\{ \left\Vert \xi\right\Vert >1\right\} \right)+\frac{2}{c_{2}}\int_{\left\{ \left\Vert \xi\right\Vert >1\right\} }\log\left\Vert \xi\right\Vert m\left(\mathrm{d}\xi\right).}
\end{align*}
Because the Lévy measure $m(\mathrm{d}\xi)$ integrates $\mathbbm{1}_{\left\{ \left\Vert \xi\right\Vert >1\right\} }\log\Vert\xi\Vert$
by assumption, we see that
\begin{equation}
\mathcal{I}^{\ast}(u)<\infty.\label{eq: Lemma 4.1, I*}
\end{equation}

We now turn to $\mathcal{I}_{\ast}(\xi)$. We can write
\begin{align*}
\mathrm{e}^{\langle\xi,\psi(s,u)\rangle} & -1-\langle\psi^{J}\left(s,u\right),\xi_{J}\rangle\\
 & =\int_{0}^{1}\mathrm{e}^{r\langle\xi,\psi(s,u)\rangle}\langle\psi\left(s,u\right),\xi\rangle\mathrm{d}r-\langle\psi^{J}\left(s,u\right),\xi_{J}\rangle\\
 & =\int_{0}^{1}\mathrm{e}^{r\langle\xi,\psi(s,u)\rangle}\langle\psi^{I}\left(s,u\right),\xi_{I}\rangle\mathrm{d}r+\int_{0}^{1}\left(\mathrm{e}^{r\langle\xi,\psi(s,u)\rangle}-1\right)\langle\psi^{J}\left(s,u\right),\xi_{J}\rangle\mathrm{d}r\\
 & =\int_{0}^{1}\mathrm{e}^{r\langle\xi,\psi(s,u)\rangle}\langle\psi^{I}\left(s,u\right),\xi_{I}\rangle\mathrm{d}r\\
 & \qquad+\int_{0}^{1}\int_{0}^{1}\mathrm{e}^{rr'\langle\xi,\psi(s,u)\rangle}r\langle\xi,\psi(s,u)\rangle\langle\psi^{J}\left(s,u\right),\xi_{J}\rangle\mathrm{d}r\mathrm{d}r'.
\end{align*}
Noting \eqref{eq:psi exp ineq} and $\mathrm{Re}\left(\langle\xi,\psi(s,u)\rangle\right)\leqslant0$,
we deduce that for $\|\xi\|\leqslant1$ and $s\geqslant0$,
\begin{align}
\left|\mathrm{e}^{\langle\xi,\psi(s,u)\rangle}-1-\langle\psi^{J}\left(s,u\right),\xi_{J}\rangle\right| & \leqslant\left\Vert \psi^{I}\left(s,u\right)\right\Vert \left\Vert \xi_{I}\right\Vert +\big\|\psi\left(s,u\right)\big\|\left\Vert \xi\right\Vert \left\Vert \psi^{J}\left(s,u\right)\right\Vert \left\Vert \xi_{J}\right\Vert \nonumber \\
 & \leqslant(c_{1}+c_{1}^{2})\mathrm{e}^{-c_{2}s}\left(\left\Vert \xi_{I}\right\Vert +(\left\Vert \xi_{I}\right\Vert +\left\Vert \xi_{J}\right\Vert )\left\Vert \xi_{J}\right\Vert \right)\nonumber \\
 & \leqslant(c_{1}+c_{1}^{2})\mathrm{e}^{-c_{2}s}\left(2\left\Vert \xi_{I}\right\Vert +\left\Vert \xi_{J}\right\Vert ^{2}\right).\label{eq: esti for I_*}
\end{align}
 So\textcolor{red}{
\begin{align*}
{\normalcolor \mathcal{I}_{\ast}\left(u\right)} & {\normalcolor \leqslant(c_{1}+c_{1}^{2})\int_{0}^{\infty}\mathrm{e}^{-c_{2}s}\mathrm{d}s\int_{\left\{ 0<\left\Vert \xi\right\Vert \leqslant1\right\} }\left(2\left\Vert \xi_{I}\right\Vert +\left\Vert \xi_{J}\right\Vert ^{2}\right)m\left(\mathrm{d}\xi\right)<\infty,}
\end{align*}
}where the finiteness of the integral on the right-hand side follows
by Definition \ref{def:admissible parameters} (iii). Since \eqref{eq: Lemma 4.1, I*}
holds, it follows that
\[
\int_{0}^{\infty}\left|F\left(\psi\left(s,u\right)\right)\right|\mathrm{d}s\leqslant\mathcal{I}\left(u\right)=\mathcal{I}_{\ast}\left(u\right)+\mathcal{I}^{\ast}\left(u\right)<\infty.
\]
The lemma is proved.
\end{proof}
We are now ready to prove our main result.

\begin{proof}[Proof of Theorem \ref{thm:stationarity of affine processes}]
Recall that the characteristic function of $X_{t}$ is given by
\[
\mathbb{E}_{x}\left[\mathrm{e}^{\langle u,X_{t}\rangle}\right]=\exp\left\{ \phi\left(t,u\right)+\langle x,\psi\left(t,u\right)\rangle\right\} ,\quad\left(t,u\right)\in\mathbb{R}_{\geqslant0}\times\mathcal{U}.
\]
Using Remark \ref{rem:The-assumption-that}, Theorem \ref{thm:exponential convergence of psi for all u}
and Lemma \ref{lem:finiteness of limit F(psi(s,u))}, we have that
$\psi(t,u)\to0$ and
\[
\phi(t,u)=\int_{0}^{t}F\left(\psi(s,u)\right)\mathrm{d}s\to\int_{0}^{\infty}F\left(\psi(s,u)\right)\mathrm{d}s,\quad\text{as \ensuremath{t\to\infty.}}
\]
We now verify that $\int_{0}^{\infty}F\left(\psi\left(s,u\right)\right)\mathrm{d}s$
is continuous at $u=0$. It is easy to see that that $\int_{0}^{T}F\left(\psi(s,u)\right)\mathrm{d}s$
is continuous at $u=0$. It suffices to show that the convergence
$\lim_{T\to\infty}\int_{0}^{T}F\left(\psi(s,u)\right)\mathrm{d}s=\int_{0}^{\infty}F\left(\psi(s,u)\right)\mathrm{d}s$
is uniform for $u$ in a small neighborhood of $0$. By \eqref{eq0: lemma 3.8},
there exist $\delta>0$ and constants $c_{1},\thinspace c_{2}>0$
such that for all $B_{\delta}(0)\cap\mathcal{U}$,
\[
\left\Vert \psi\left(t,u\right)\right\Vert \leqslant c_{1}\exp\left\{ -c_{2}t\right\} ,\quad t\geqslant0.
\]
Define
\begin{align*}
\mathcal{I}_{T}\left(u\right) & =\int_{T}^{\infty}\int_{\left\{ 0<\left\Vert \xi\right\Vert \leqslant1\right\} }\left|\mathrm{e}^{\langle\xi,\psi(s,u)\rangle}-1-\langle\psi^{J}(s,u),\xi_{J}\rangle\right|m\left(\mathrm{d}\xi\right)\mathrm{d}s\\
 & \quad+\int_{T}^{\infty}\int_{\left\{ 1<\left\Vert \xi\right\Vert \leqslant K\right\} }\left|\mathrm{e}^{\langle\xi,\psi(s,u)\rangle}-1\right|m\left(\mathrm{d}\xi\right)\mathrm{d}s\\
 & \quad+\int_{T}^{\infty}\int_{\left\{ \left\Vert \xi\right\Vert >K\right\} }\left|\mathrm{e}^{\langle\xi,\psi(s,u)\rangle}-1\right|m\left(\mathrm{d}\xi\right)\mathrm{d}s\\
 & =:\mathcal{I}_{\ast,T}\left(u\right)+\mathcal{I}_{T}^{\ast}\left(u\right)+\mathcal{I}_{T}^{\ast\ast}\left(u\right),
\end{align*}
where $K>0$. Let $\varepsilon>0$ be arbitrary. By Fubini's theorem,
\begin{align*}
\mathcal{I}_{T}^{\ast\ast}\left(u\right) & =\int_{\left\{ \left\Vert \xi\right\Vert >K\right\} }\int_{T}^{\infty}\left|\mathrm{e}^{\langle\xi,\psi(s,u)\rangle}-1\right|\mathrm{d}sm\left(\mathrm{d}\xi\right).
\end{align*}
Set $I_{1}\left(\xi\right):=\int_{T}^{\infty}\left|\exp\lbrace\langle\psi(s,u),\xi\rangle\rbrace-1\right|\mathrm{d}s$.
As in the proof of Lemma \ref{lem:finiteness of limit F(psi(s,u))},
we introduce a change of variables $t:=\exp\left\{ -c_{2}(s-T)\right\} \left\Vert \xi\right\Vert $
and obtain for $\|\xi\|>1$,
\begin{align}
I_{1}\left(\xi\right) & =\frac{1}{c_{2}}\int_{0}^{\left\Vert \xi\right\Vert }\frac{1}{t}\left|\mathrm{e}^{\langle\xi,\psi\left(s^{-1}(t),u\right)\rangle}-1\right|\mathrm{d}t\label{eq: esti for I_1(xi)}\\
 & \leqslant\frac{1}{c_{2}}\int_{0}^{1}\frac{1}{t}\left|\mathrm{e}^{\langle\xi,\psi\left(s^{-1}(t),u\right)\rangle}-1\right|\mathrm{d}t+\frac{1}{c_{2}}\int_{1}^{\left\Vert \xi\right\Vert }\frac{2}{t}\mathrm{d}t\nonumber \\
 & \leqslant\frac{1}{c_{2}}\int_{0}^{1}\frac{c_{1}}{t}\mathrm{e}^{-c_{2}s^{-1}(t)}\left\Vert \xi\right\Vert \mathrm{d}t+\frac{2}{c_{2}}\log\left\Vert \xi\right\Vert \nonumber \\
 & \leqslant\frac{1}{c_{2}}\int_{0}^{1}c_{1}e^{-c_{2}T}\mathrm{d}t+\frac{2}{c_{2}}\log\left\Vert \xi\right\Vert .\nonumber
\end{align}
So
\begin{align*}
\mathcal{I}_{T}^{\ast\ast}\left(u\right) & \leqslant\int_{\left\{ \left\Vert \xi\right\Vert >K\right\} }\left(\frac{c_{1}}{c_{2}}e^{-c_{2}T}+\frac{2}{c_{2}}\log\left\Vert \xi\right\Vert \right)m\left(\mathrm{d}\xi\right)\\
 & \leqslant\frac{c_{1}}{c_{2}}m\left(\left\{ \left\Vert \xi\right\Vert >K\right\} \right)+\frac{2}{c_{2}}\int_{\left\{ \left\Vert \xi\right\Vert >K\right\} }\log\left\Vert \xi\right\Vert m\left(\mathrm{d}\xi\right).
\end{align*}
We now choose $K>0$ large enough such that $\mathcal{I}_{T}^{\ast\ast}\left(u\right)<\varepsilon/3$.

For $\mathcal{I}_{T}^{\ast}\left(u\right)$, by (\ref{eq: esti for I_1(xi)}),
we have
\begin{align*}
I_{1}\left(\xi\right) & =\frac{1}{c_{2}}\int_{0}^{\left\Vert \xi\right\Vert }\frac{1}{t}\left|\mathrm{e}^{\langle\xi,\psi\left(s^{-1}(t),u\right)\rangle}-1\right|\mathrm{d}t\\
 & \leqslant\frac{1}{c_{2}}\int_{0}^{\left\Vert \xi\right\Vert }\frac{c_{1}}{t}\mathrm{e}^{-c_{2}s^{-1}(t)}\left\Vert \xi\right\Vert \mathrm{d}t\\
 & \leqslant\frac{1}{c_{2}}\int_{0}^{\left\Vert \xi\right\Vert }c_{1}e^{-c_{2}T}\mathrm{d}t\\
 & \leqslant\frac{c_{1}}{c_{2}}e^{-c_{2}T}\left\Vert \xi\right\Vert ,
\end{align*}
which imples
\begin{align*}
\mathcal{I}_{T}^{\ast}\left(u\right) & \leqslant\int_{\left\{ 1<\left\Vert \xi\right\Vert \le K\right\} }\left(\frac{c_{1}}{c_{2}}e^{-c_{2}T}\left\Vert \xi\right\Vert \right)m\left(\mathrm{d}\xi\right)\\
 & \leqslant\frac{c_{1}}{c_{2}}e^{-c_{2}T}\int_{\left\{ 1<\left\Vert \xi\right\Vert \le K\right\} }\left\Vert \xi\right\Vert m\left(\mathrm{d}\xi\right)\to0,\quad\mbox{as \ensuremath{T\to\infty.}}
\end{align*}
So we find $T_{1}>0$ such that for $T>T_{1}$, $\mathcal{I}_{T}^{\ast}\left(u\right)<\varepsilon/3$.
It follows from (\ref{eq: esti for I_*}) that
\begin{align*}
\mathcal{I}_{\ast,T}\left(u\right) & \leqslant(c_{1}+c_{1}^{2})\int_{T}^{\infty}\mathrm{e}^{-c_{2}s}\mathrm{d}s\int_{\left\{ 0<\left\Vert \xi\right\Vert \leqslant1\right\} }\left(2\left\Vert \xi_{I}\right\Vert +\left\Vert \xi_{J}\right\Vert ^{2}\right)m\left(\mathrm{d}\xi\right)\to0,\quad\mbox{as \ensuremath{T\to\infty}}.
\end{align*}
Hence there exists $T_{2}>T_{1}$ such that for $T>T_{2}$, $\mathcal{I}_{\ast,T}\left(u\right)<\varepsilon/3$.
Finally, we get for $T>T_{2}$,
\[
\int_{T}^{\infty}\left|F\left(\psi\left(s,u\right)\right)\right|\mathrm{d}s\leqslant\mathcal{I}_{\ast,T}\left(u\right)+\mathcal{I}_{T}^{\ast}\left(u\right)+\mathcal{I}_{T}^{\ast\ast}\left(u\right)<\varepsilon.
\]
Moreover, the particular choice of above $K,T_{1},T_{2}$ do not depend
on $u\in B_{\delta}(0)\cap\mathcal{U}$. We thus obtain the desired
uniform convergence and further the continuity of $\int_{0}^{\infty}F\left(\psi\left(s,u\right)\right)\mathrm{d}s$
at $u=0$.

By Lévy's continuity theorem, the limiting distribution of $X_{t}$
exists and we denote it by $\pi$. The limiting distribution $\pi$
has characteristic function
\[
\int_{D}\mathrm{e}^{\langle u,x\rangle}\pi\left(\mathrm{d}x\right)=\exp\left\{ \int_{0}^{\infty}F\left(\psi(s,u)\right)\mathrm{d}s\right\} .
\]

We now verify that $\pi$ is the unique stationary distribution. We
start with the stationarity. Suppose that $X_{0}$ is distributed
according to $\pi$. Then, for any $u\in\mathcal{U}$,
\begin{align*}
\mathbb{E}_{\pi}\left[\exp\left\{ \langle u,X_{t}\rangle\right\} \right] & =\int_{D}\exp\left\{ \phi(t,u)+\langle x,\psi(t,u)\rangle\right\} \pi(\mathrm{d}x)\\
 & =\mathrm{e}^{\phi(t,u)}\int_{D}\exp\left\{ \langle x,\psi(t,u)\rangle\right\} \pi(\mathrm{d}x)\\
 & =\mathrm{e}^{\phi(t,u)}\int_{D}\mathrm{e}^{\langle x,\eta\rangle}\pi(\mathrm{d}x),
\end{align*}
where we substituted $\eta:=\psi(t,u)$ in the last equality. Note
that the integral on the right-hand side of the last equality is the
characteristic function of the limit distribution $\pi$. Therefore,
using the semi-flow property of $\psi$ in \eqref{eq:semi flow property},
we have
\begin{align*}
\mathbb{E}_{\pi}\left[\exp\left\{ \langle u,X_{t}\rangle\right\} \right] & =\mathrm{e}^{\phi(t,u)}\exp\left\{ \int_{0}^{\infty}F\left(\psi(s,\eta)\right)\mathrm{d}s\right\} \\
 & =\mathrm{e}^{\phi(t,u)}\exp\left\{ \int_{0}^{\infty}F\left(\psi\left(s,\psi(t,u)\right)\right)\mathrm{d}s\right\} \\
 & =\mathrm{e}^{\phi(t,u)}\exp\left\{ \int_{0}^{\infty}F\left(\psi(t+s,u)\right)\mathrm{d}s\right\} \\
 & =\mathrm{e}^{\phi(t,u)}\exp\left\{ \int_{t}^{\infty}F\left(\psi(s,u)\right)\mathrm{ds}\right\} .
\end{align*}
So, by the generalized Riccati equation \eqref{eq:riccati equation for psi ell}
for $\phi$,
\[
\mathbb{E}_{\pi}\left[\exp\left\{ \langle u,X_{t}\rangle\right\} \right]=\exp\left\{ \int_{0}^{\infty}F\left(\psi(s,u)\right)\mathrm{d}s\right\} =\int_{D}\mathrm{e}^{\langle x,u\rangle}\pi(\mathrm{d}x).
\]
Hence $\pi$ is a stationary distribution for $X$.

Finally, we prove the uniqueness of stationary distributions for $X$.
We proceed as in \cite[p.80]{MR2779872}. Suppose that there exists
another stationary distribution $\pi'$. Let $X_{0}$ be distributed
according to $\pi'$. Recall that for all $u\in\mathcal{U}$, $\psi(t,u)\to0$
as $t\to\infty$ in virtue of Theorem \ref{thm:exponential convergence of psi for all u}
and, by Lemma \ref{lem:finiteness of limit F(psi(s,u))}, $\phi(t,u)\to\int_{0}^{\infty}F\left(\psi(t,u)\right)\mathrm{d}s$
as $t\to\infty$. Hence, by dominated convergence,
\begin{align*}
\int_{D}\mathrm{e}^{\langle x,u\rangle}\pi'(\mathrm{d}x) & =\lim_{t\to\infty}\mathbb{E}_{\pi'}\left[\exp\left\{ \langle u,X_{t}\rangle\right\} \right]\\
 & =\lim_{t\to\infty}\int_{D}\exp\left\{ \phi(t,u)+\langle x,\psi(t,u)\rangle\right\} \pi'(\mathrm{d}x)\\
 & =\int_{D}\exp\left\{ \int_{0}^{\infty}F\left(\psi(s,u)\right)\mathrm{d}s\right\} \pi'(\mathrm{d}x)\\
 & =\exp\left\{ \int_{0}^{\infty}F\left(\psi(s,u)\right)\mathrm{d}s\right\} =\int_{D}\mathrm{e}^{\langle x,u\rangle}\pi(\mathrm{d}x).
\end{align*}
So $\pi=\pi'$.\end{proof}
\begin{acknowledgement*}
We would like to thank Martin Friesen for several helpful discussions.
\end{acknowledgement*}
\bibliographystyle{amsplain}
\addcontentsline{toc}{section}{\refname}

\def\cprime{$'$} \def\cprime{$'$}
\providecommand{\bysame}{\leavevmode\hbox to3em{\hrulefill}\thinspace}
\providecommand{\MR}{\relax\ifhmode\unskip\space\fi MR }
\providecommand{\MRhref}[2]{%
  \href{http://www.ams.org/mathscinet-getitem?mr=#1}{#2}
}
\providecommand{\href}[2]{#2}

\end{document}